\documentclass[12pt]{amsart}

\usepackage[margin=25mm]{geometry}

\usepackage{amsmath, amsthm, amsfonts, amssymb, cite, enumitem, nccmath}
\usepackage{mathrsfs}
\usepackage{hyperref}
\usepackage{color,hyperref}

\definecolor{darkblue}{rgb}{0.0,0.0,0.5}
\definecolor{darkred}{rgb}{0.5,0.0,0.0}

\hypersetup{colorlinks,breaklinks,linkcolor=darkblue,urlcolor=darkred,anchorcolor=darkblue,citecolor=darkblue}

\newtheorem{theorem}{Theorem}[section]
\newtheorem{proposition}[theorem]{Proposition}
\newtheorem{corollary}[theorem]{Corollary}
\newtheorem{lemma}[theorem]{Lemma}
\newtheorem{rk}[theorem]{Remark}

\newtheorem{defin}[theorem]{Definition}

\newcommand{\ze}{\mathbb{Z}}
\newcommand{\re}{\mathbb{R}}

\newcommand{\p}{\mathbb{P}}

\def\bR{\mathbb{R}}
\def\bZ{\mathbb{Z}}

\newcommand{\esp}{\mathbb{E}}

\def\para{\theta}
\newcommand{\wh}{\mathscr{W}}
\newcommand{\cB}{\mathcal{B}}
\newcommand{\she}{\mathcal{Z}}
\newcommand{\kpz}{\mathcal{H}}
\newcommand{\ou}{\mathcal{U}}
\newcommand{\kpzfp}{\mathbf{H}}

\newcommand{\cS}{\mathcal{S}}
\newcommand{\cR}{\mathcal{R}}

\newcommand{\cRKPZ}{\mathcal{R}^{\textsc{KPZ}}}
\newcommand{\cRKPZeq}{\mathcal{R}^{\textsc{KPZeq}}}

\newcommand{\cRLPP}{\mathcal{R}^{\textsc{LPP}}}
\newcommand{\cRexcl}{\mathcal{R}^{\textsc{TASEP}}}

\newcommand{\rhoKPZ}{\rho^{\textsc{KPZ}}}
\newcommand{\rhoEW}{\rho^{\textsc{EW}}}

\newcommand{\grad}{\nabla}

\newcommand{\dd}{\mathrm{d}}

\def\Vvv{{\rm\mathbb{V}ar}}
\def\Cvv{{\rm\mathbb{C}ov}}
\def\CRvv{{\rm\mathbb{C}orr}}

\newcommand{\Cal}[1]{\mathcal{#1}}

\author[J.-D.\ Deuschel G.\ Moreno Flores T.\ Orenshtein]{Jean-Dominique Deuschel$^\star$, Gregorio R.\ Moreno Flores$^\diamond$ and Tal Orenshtein$^\triangledown$}
\thanks{$^\star$ TU Berlin, deuschel@math.tu-berlin.de}
\thanks{$^\diamond$ Pontificia Universidad Cat\'olica de Chile, grmoreno@mat.uc.cl}
\thanks{$^\triangledown$ TU Berlin and Weierstrass Institute Berlin, orenshtein@wias-berlin.de}

\date{}

\title[Aging for stationary KPZ]{Aging for the stationary Kardar-Parisi-Zhang equation and related models}

\begin{document}

\begin{abstract}
We study the aging property for stationary models in the KPZ universality class. In particular, we show aging for the stationary KPZ fixed point, the Cole-Hopf solution to the stationary KPZ equation, the height function of the stationary TASEP, last-passage percolation with boundary conditions and stationary directed polymers in the intermediate disorder regime. All of these models are shown to display a universal aging behavior characterized by the rate of decay of their correlations.
As a comparison, we show aging for models in the Edwards-Wilkinson universality class where a different decay exponent is obtained.
A key ingredient to our proofs is a characteristic of space-time stationarity - covariance-to-variance reduction - which allows to deduce the asymptotic behavior of the correlations of two space-time points by the one of the variances at one point.
We formulate several open problems.
\end{abstract}

\maketitle

\tableofcontents


\section{Introduction}\label{sec:intro}


%
%
Aging is a property satisfied by a wide family of non-equilibrium dynamics in disordered
media, including many interesting processes in random environments. Heuristically, a process is said to satisfy aging if, the older it gets, the longer it takes to forget its past.
To properly state this property, we consider the correlation between two random variables $Q_1$ and $Q_2$:
	\begin{align*}
    \CRvv(Q_1,Q_2)= \frac{\Cvv\, (Q_1,Q_2)}{\sqrt{\Vvv\, Q_1} \sqrt{ \Vvv \, Q_2}}.
    \end{align*}
We say that a process $(Y_t)_{t\ge 0}$ satisfies the \emph{aging} property with respect to the aging function $\rho$ if
$\rho:[1,\infty) \to (0,1]$ is so that $\rho(a)< 1$  for $a> 1$ and 
\begin{align*}
	\lim_{t\to\infty}\CRvv\left( Y_t,\, Y_{at}\right)
	=
	\rho(a)
\end{align*}
for all $a\ge1$.

The study of aging originated in the physics literature in the context of cooling experiments for glassy systems \cite{Struick, Ferry}. Here, the age of the system was understood as the time during which the system is kept at a fixed temperature. It was observed that older systems take longer to relax when submitted to a thermal variation.

The problem then reached the mathematics community, especially in the study of spin glasses and trap models \cite{BC1, BC2, BBG, BCM, BDG}.
Aging was studied also for random walks in random environment, with yet a similar formulation in which the correlation function is replaced by the distance distribution function, cf.\ \cite{ESZ,DGZ} and the references therein.

The aging property was observed in interacting diffusions as well \cite{DD}. One of these is the parabolic Anderson model (PAM)
\begin{eqnarray*}
	dZ_j(t) = \tfrac12 \Delta Z_j(t)\dd t + \beta Z_j(t) dB_j(t), \quad t\geq 0,\, j\in\ze^d,
\end{eqnarray*}
where $\beta>0$, $\Delta$ is the discrete Laplacian and $(B_j)_{j}$ is an i.i.d.\ family of standard one-dimensional Brownian motions. It was noticed that such a property is highly sensitive to the details of the models. 
In particular, it is shown in \cite[Proposition 1.3 (iii)]{DD} that no aging takes place
for $Z_j$. On the other hand, it was conjectured that aging
should take place for $\log Z_j$
in $(d+1)$-dimensions with $d=1,2$ and for large $\beta$ in dimensions $d\ge3$. This conjecture was the original motivation for our work. 

The PAM can be seen as a discrete version of the Stochastic Heat Equation with multiplicative noise (SHE)
(cf.\ \eqref{eq:she} below). The SHE is in turn related to the Kardar-Parisi-Zhang equation (KPZ) through the Cole-Hopf transformation $\mathcal H(t,x)=\log \mathcal Z(t,x)$, where $\she$ is the solution of the SHE. One of our main results - Theorem \ref{thm:main-KPZ-FP} - shows aging for the KPZ equation in 1$+$1 dimensions in the {\it stationary}
regime, i.e.\ when $\kpz(t,0)=\cB(x)$, where $\cB$ is a two-sided Brownian motion (see Section \ref{sec:summary} for a precise definition of stationarity). Moreover, this is obtained with an \emph{explicit} aging function.

The study of aging in the KPZ universality class has been the object of many recent works in experimental physics \cite{TS, dNlDT}, theoretical physics \cite{FS16, dNlD17, lD17, dNlD18} and mathematics \cite{BG18,FO19,CGH20}.
It is intimately related to the two-time correlations of the models under consideration, a challenging problem which has been successfully tackled only in the last few years. In this work, we show aging for several stationary models in this class, always with the same aging function. This can be seen as a lower resolution observation on the sensitivity of the property to the details of the models; it supports the idea that the aging behavior is yet another universal property inside the KPZ universality class.


\subsection{The KPZ universality class}\label{sec:previous}


The Kardar-Parisi-Zhang equation was introduced in the physics literature as a model of phase separation lines in the presence of impurities \cite{KPZ}.
It can be written as
\begin{eqnarray*}
	\partial_t \kpz = \tfrac12 \partial_x^2 \kpz + |\partial_x \kpz|^2 + \wh,
\end{eqnarray*}
where $\wh$ is a space-time white noise.
We refer to \cite{Q-review, C-review, SQ} for comprehensive reviews on this equation and the KPZ universality class, including physical systems under its scope. Among the many possible initial conditions, three of them have attracted much attention: the narrow wedge initial condition $e^{\kpz(0,\cdot)}=\delta_0(\cdot)$, the flat initial condition $\kpz(0,\cdot)=0$ and the stationary initial condition $\kpz(0,\cdot)=\cB(\cdot)$, where $\cB$ is a two-sided Brownian motion.

The fluctuations of the KPZ equation are fairly well understood. It is known that there exist constants $v_{\infty}, c \in \re$ such that
\begin{eqnarray*}
	\frac{
		\kpz(0,t) - v_{\infty} t
	}
	{
		c t^{1/3}
	}
	\Rightarrow
	\chi,
	\quad
	t\to\infty,
\end{eqnarray*}
where the distribution of $\chi$ is not Gaussian and moreover depends on the initial conditions. In particular, it has been identified as the GUE distribution for the narrow-wedge initial condition, the GOE distribution for the flat initial condition and the Baik-Rains distribution for stationary initial conditions. Such behavior is shared by many relevant models in the KPZ universality class such as last-passage percolation models and the totally asymmetric simple exclusion process.

The study of aging properties requires information on two-time correlation functions. Persistence of memory was first observed in the experimental setting in \cite{TS} in the framework of spectacular liquid-crystal turbulence experiments. Aging was observed for the phase separation line of such systems on circular (narrow-wedge) and flat substrates. Such behaviour was further confirmed in \cite{dNlDT} as well as aging in numerical simulations of the Eden model.

There has been historically at least two different approaches to the computation of the two-point correlations of models in the KPZ universality class. The first one consists in expressing the correlations in terms of simpler Airy-like processes. This approach was pioneered in \cite{FS16} for last-passage percolation for narrow-wedge, flat and stationary initial profiles.
In the latter case, it was conjectured that
\begin{eqnarray}\label{eq:FS-variational}
	\lim_{n\to\infty}\Cvv\left( L(n),\, L(an) \right)
	=
	\tau^{2/3}
	\Cvv\left(
		\mathcal{A}(0),\,
		\max_{u\in\re}\{
			\mathcal{A}(u)
			+
			\hat{a}^{-1/3} \tilde{\mathcal{A}}(u\hat{a}^{2/3})-u^2 \hat{a}
		\}
	\right),
\end{eqnarray}
where $L$ denotes the passage time (see Section \ref{sec:LPP} for a precise definition of the model). In the above formula, $\mathcal{A}$ and $\tilde{\mathcal{A}}$ are two independent Airy processes and $\hat{a}=(a-1)^{-1}$. It was further conjectured that
\begin{eqnarray}\label{eq:FS-correlation}
	\lim_{n\to\infty}n^{-2/3}\Cvv\left( L(n),\, L(an) \right)
	=
	C
	\left(
		1+a^{2/3}-(a-1)^{2/3}
	\right),
\end{eqnarray}
where $C$ is identified as the variance of the Baik-Rains distribution \cite[Formula 2.6]{FS16}. {The variational formula \eqref{eq:FS-variational} was later proved in \cite{FO19} while we prove \eqref{eq:FS-correlation} in Theorem \ref{thm:main-LPP}. Note that it was proved in \cite{FO19} that the limiting covariance in \eqref{eq:FS-variational} can be expressed as a combination of variances of Airy process. We obtain such identities at finite scales for all the stationary processes considered in this work (see Section \ref{sec:warm-up}).}

The second approach consists in obtaining highly non-trivial formulae for the two-point distribution function. Conjectural formulae were obtained in \cite{dNlD17, dNlD18} for the KPZ equation with narrow-wedge initial condition and in \cite{lD17} for the Airy process minus a parabola plus a Brownian motion. Rigorous formulae were obtained for a continuum last-passage percolation model in \cite{J1} and for last-passage percolation with geometric weights in an appropriate scaling limit in \cite{J2}. Long and short time asymptotics were then obtained in \cite{J3}. Note that no asymptotics for the correlations are provided in these works. In general, obtaining such information from exact formulae involves very refined asymptotic analysis.

We finally comment on two recent works. In \cite{BG18}, the authors obtain bounds on the correlations of the last-percolation model with exponential weights which match the predictions of \cite{FS16}. The work \cite{CGH20} provides such bounds for the KPZ equation with narrow wedge initial condition, providing the first rigorous results for a positive-temperature model (in contrast with the zero-temperature nature of last-passage percolation). More precisely, if
\begin{align*}
	\overline{\kpz}_t(a,x)
	=
	t^{-1/3}
	\left(
		\kpz(a t, t^{2/3}x) + \alpha t/24
	\right),
\end{align*}
then, it is showed that there exists two positive constants $c_1,\, c_2$ and $t_0>0$ such that
\begin{align*}
	c_1 a^{-1/3}
	\leq
	\CRvv\left(
		\overline{\kpz}_t(1,0),\, \overline{\kpz}_t(a,0)
	\right)
	&
	\leq
	c_2 a^{-1/3},
	\quad
	&
	\forall \, a>2,\, t>t_0,
	\\
	c_1 (a-1)^{2/3}
	\leq
	1-
	\CRvv\left(
		\overline{\kpz}_t(1,0),\, \overline{\kpz}_t(a,0)
	\right)
	&
	\leq
	c_2 (a-1)^{2/3},
	\quad
	&
	\forall \, a\in(1,\tfrac32),\, t>\tfrac{t_0}{a-1}.
\end{align*}
Our work complements the above results by providing an explicit aging function in the stationary case for both last-passage percolation and the KPZ equation.


\subsection{Summary of results}\label{sec:summary}


In this paper we address the question of aging for stationary models in $1+1$ dimensions. A random process $\{F(t,x):\, t\ge0,\, x\in \bR\}$, or $\{F(t,x):\, t\ge0,\, x\in \bZ\}$, is \emph{stationary} (resp.\ \emph{space-time stationary}) if the law of its space (resp.\ space-time) increments is invariant under shifting the process in
time (resp.\ in space-time). A formal definition of space-time stationarity can be found in the content of Proposition \ref{prop:st_stat_KPZ_FP}.
The notion of stationarity should not be confused with equilibrium, in which case the law of the process itself is invariant under time shifts. All our results follow a similar scheme which is outlined in Section \ref{sec:warm-up} and starts by reducing the computation of covariances to a combination of variances involving the processes at a single space-time point.

We first deduce an aging property characterized by the aging function
\begin{align*}
	\rhoKPZ(a)
	:=
	\frac{1+a^{2/3}-(a-1)^{2/3}}{2a^{1/3}}
\end{align*}
for models in the KPZ universality class.
We refer to $\rhoKPZ$ as the KPZ aging function. While this is an asymptotic result for most of the models, the two-time correlations of the KPZ fixed point\cite{MQR} are shown to be exactly given by $\rhoKPZ$. This follows easily from the fact that this process  satisfies the $3:2:1$ scaling property which is a characteristic of its universality class (but is not satisfied by microscopic models at a fixed scale).
The KPZ fixed point is the central object of its own universality class in the sense that it should arise as the scaling limit of the models in this class. This is showed to hold for the totally asymmetric exclusion process (TASEP) in \cite{MQR} and very recently for the KPZ equation in \cite{QS20} and \cite{Vir}.

We then prove aging for the Cole-Hopf solution to the KPZ equation.
Next, we prove aging for the height function of TASEP and last-passage percolation (LPP) with exponential weights and boundary conditions. Our proofs rely
crucially on further tail estimates for these models which have been derived in \cite{BFP} and \cite{CG}.

Finally, we show aging for directed polymers in Brownian environment in the intermediate disorder regime (O'Connell-Yor model \cite{OY}). The argument, which is the most technical part of this work, uses Talagrand's concentration method in Malliavin formalism and may be of independent interest. In this case, we show that the correlations of the model rescale to the ones of the KPZ equation.

As a comparison, we also consider the Edwards-Wilkinson (EW) universality class, which has scaling $4:2:1$. The aging behavior in this case is characterized by the aging function
\begin{align*}
	\rhoEW(a)
	=
	\frac{1+a^{1/2}-(a-1)^{1/2}}{2a^{1/4}}.
\end{align*}
We call $\rhoEW$ the EW aging function.  Once again, the stationary EW model (AKA the stochastic heat equation with additive noise) is at the center of this universality class and its two-time correlations are exactly given by $\rhoEW$. For one-dimensional gradient models which rescale to the EW model, we generalize aging and show the convergence of the rescaled space-time correlation function to the one of EW.

\subsection{A warm up}\label{sec:warm-up}

All the models considered hereafter are stationary.
Our methods follow a fairly elementary scheme and are based on the simple observation that if a process $(Y_t)_{t\ge 0}$ is such that \[
 \Vvv (Y_{t}-Y_{s}) =  \Vvv (Y_{t-s})
\]
for all $0\le s\le t$ and, in addition, its variance satisfies
\[
\lim_{s \rightarrow\infty}\frac{ \Vvv (Y_{s})}{s^{2\alpha}}  = v
\]
for some $v,\alpha>0$, then
\[
\lim_{s\to\infty}\CRvv(Y_s,Y_{as}) = \frac{1+a^{2\alpha}-(a-1)^{2\alpha}}{2a^\alpha}
\]
for all $a\ge 1$. This is a direct consequence of a property we call covariance-to-variance reduction which is detailed in the definition below and allows us to express the covariance as a combination of variances.
In Lemma \ref{lem:stat-CVTV}, we prove the elementary fact that stationary processes satisfy this property.
We stress that the convergence of the variances is a highly non-trivial fact for most of the models considered in this work.

\begin{defin}\label{def:CVTV}
 A random process $F(t,x)$, $t\ge0$, $x\in S$, for $S=\bR$ or $S=\bZ$, is said to follow the \emph{covariance-to-variance reduction} if for all $0\leq t_1 < t_2$ and $x_1,x_2\in S$, it holds that
	\begin{align*}
	\CRvv\left( F(t_1,x_1),\, F(t_2,x_2)\right)		&=
		\frac12
		\frac{\Vvv\, F(t_1,x_1)+\Vvv\, F(t_2,x_2)-\Vvv\, F(t_2-t_1,x_2-x_1)}{\sqrt{\Vvv\, F(t_1,x_1) \, \Vvv \,F(t_2,x_2)}}.
	\end{align*}
\end{defin}

\begin{lemma}\label{lem:stat-CVTV}
 A space-time stationary process follows the covariance-to-variance reduction.
\end{lemma}
\begin{proof}
Applying the elementary identity
\begin{align*}
	\Vvv \, (Q_1-Q_2)
	=
	\Vvv \, Q_1 + \Vvv \, Q_2 - 2 \, \Cvv\, ( Q_1,Q_2 )
\end{align*}
for two random variables $Q_1$ and $Q_2$,
we get
\begin{align*}
	\Cvv\left( F(s,x),F(t,y)\right)
	=
	\frac12\big(\Vvv \, F(s,x) + \Vvv \, F(t,y)-\Vvv \left( F(t,y)-F(s,x)\right)\big).
\end{align*}
Assuming that $\Vvv \left( F(t,y)-F(s,x)\right) = \Vvv\, F(t-s,y-x)$, the above becomes
\begin{align*}
	\Cvv\,\left( F(s,x),F(t,y)\right)
	=
	\frac12
	\big(\Vvv \, F(s,x) + \Vvv \, F(t,y)-\Vvv\left( F(t-s,y-x)\right)\big),
\end{align*}
which proves the covariance-to-variance reduction for $F$.
\end{proof}
 In fact, the above proof shows that  a process is satisfying the covariance-to-variance reduction \emph{if and only if} the variance of its space-time increments is invariant under space-time shifts. All the processes considered in this paper are posteriorly space-time stationary. It might be interesting to construct non-trivial examples of processes which are not space-time stationary but exhibit the covariance-to-variance reduction.
Notice that such a process cannot be Gaussian since the latter is determined by the variances.


\section{The KPZ universality class}




\subsection{The KPZ fixed point}


The stationary KPZ fixed point $\kpzfp$ was introduced in \linebreak \cite{MQR}. In the same paper, it is proved to satisfy the $3:2:1$ scaling identity  $\big(s^{-1/3}\kpzfp(st,s^{3/2}x)\big)_{t\ge0,x\in\bR}\overset{\mathrm{law}}{=}\big(\kpzfp(t,x)\big)_{t\ge0,x\in\bR}$ for $s>0$. 
In addition, the process $\kpzfp$ is stationary, and moreover $x\mapsto \kpzfp(t,x+y)-\kpzfp(t,y)$ is a two-sided Brownian motion for all $t\geq 0$.


\begin{lemma}\label{thm:CVTV-KPZ-FP}
The KPZ fixed point $\kpzfp$ follows the covariance-to-variance reduction.
\end{lemma}
Define
\begin{align*}
	\cRKPZ(s,t;x,y)
	=
	\CRvv\left( \kpzfp(s,x),\, \kpzfp(t,y)\right)
\end{align*}
and
\begin{align*}
	\rho^{\text{KPZ}}(a)
	=
	\frac{1+a^{2/3}-(a-1)^{2/3}}{2a^{1/3}}.
\end{align*}
From the $3:2:1$ scaling identity, we immediately obtain
\begin{align*}
	\cRKPZ(s,as;0,0)
	=
	\rho^{\text{KPZ}}(a).
\end{align*}
For general end-points:
\begin{theorem}\label{thm:main-KPZ-FP}
	For all $x_1,x_2\in\re$, we have
	\begin{align*}
		\lim_{s\to\infty}
		\cRKPZ(s,as;x_1,x_2)
		=
		\rho^{\text{KPZ}}(a).
	\end{align*}
\end{theorem}

\begin{rk}
The proof of Theorem \ref{thm:main-KPZ-FP}, in Section 5 below, covers also the case $x_1,x_2=o(s^{2/3})$. For $x_1=\hat{x}_2s^{2/3}, x_2=\hat{x}_2s^{2/3}$, the scaling identity gives
\begin{align}\label{eq:remark3.2}
	\cRKPZ(st_1,st_2; s^{2/3}\hat{x}_1,s^{2/3}\hat{x}_2)
	=
	\cRKPZ(t_1,t_2; \hat{x}_1,\hat{x}_2).
\end{align}
\end{rk}


\subsection{The stationary Kardar-Parisi-Zhang equation}\label{sec:results-SHE}

We consider the stationary stochastic heat equation with multiplicative noise (SHE) i.e.\ the mild solution to the equation
\begin{eqnarray}\label{eq:she}
	\partial_t \she &=& \mfrac12 \partial_x^2 \she + \she \wh,
	\\
\nonumber
	\she(0,x) &=& e^{\cB(x)},
\end{eqnarray}
where $\wh$ is a space-time white noise and $\cB$ is a two-sided Brownian motion, that is $\cB(0)=0$ and $\{\cB(x),x\ge0\}$ and $\{\cB(-x),x\ge0\}$ are two independent standard Brownian motions. The precise meaning of a mild solution will be given in Section~\ref{sec:proof-SHE-main}.

We define $\kpz(t,x)=\log \she(t,x)$ which is commonly interpreted as the Cole-Hopf solution to the stationary KPZ equation
\begin{align*}
	\partial_t \kpz &= \mfrac12 \partial_x^2 \kpz + |\partial_x \kpz |^2 + \wh,
	\\
	\kpz(0,x) &= \cB(x).
\end{align*}
For $t_1,t_2 \geq 0$ and $x_1,x_2 \in \re$, we define
\begin{align*}
	\cRKPZeq(t_1,t_2;x_1,x_2)
	&=
	\CRvv\, (\kpz(t_1,x_1),\kpz(t_2,x_2)).
\end{align*}
\begin{theorem}\label{thm:main-SHE}
	For all $a\ge 1$ and all $x_1,x_2 \in \re$, we have
	\begin{align*}
		\lim_{s\to\infty}
		\cRKPZeq(s ,a s; x_1,x_2)
		&=
		\rhoKPZ(a).
	\end{align*}
\end{theorem}
\begin{rk}
An inspection at the proof shows that the result is still valid if we consider $x_1,x_2=o(s^{2/3})$. The case $x_1,x_2=O(s^{2/3})$ is expected to lead to a different behavior and is discussed in Subsection~\ref{sec:results-open-questions-KPZ-fixed-point} along with some other open questions.
\end{rk}
		
Once again, the proof of Theorem~\ref{thm:main-SHE} is based on the following lemma.
\begin{lemma}\label{thm:CVTV-SHE}
	The Cole-Hopf solution $\kpz$ to the KPZ equation follows the covariance-to-variance reduction.
\end{lemma}
\begin{rk}
	We easily obtain the following asymptotics for $\rhoKPZ$:
	\begin{align*}
		\rhoKPZ(a)
		&\approx
		\frac{1}{2a^{1/3}},
		\quad
		a\to\infty,
		\\
		1-\rhoKPZ(a)
		&\approx
		\mfrac12 (a-1)^{2/3},
		\quad
		a\to 1^+.
	\end{align*}
	These match the bounds from \cite{CGH20} presented at the end of Section \ref{sec:previous}.
\end{rk}

\subsection{TASEP and LPP}\label{sec:LPP}


We consider the totally asymmetric exclusion process (TASEP) on $\bZ$ with initial condition given by i.i.d.\ Bernoulli random variables with parameter $\frac12$. We denote the occupation variables by $\eta_t=\{\eta(t,j):\, t\geq 0, j\in \ze \}$ where $\eta(t,j)=1$ if there is a particle in site $j$ at time $t$ and $\eta(t,j)=0$ otherwise. We denote by $N_t(j)$ the number of particles that have jumped from $j$ to $j+1$ during the time interval $[0,t]$. We then define the height function $h$ by
\begin{align*}
	h(t,j)
	=
	\left\lbrace
		\begin{array}{ll}
			2N_t(0) - \sum^j_{\ell=1} (1-2\eta(t,\ell)), & \text{if} \, j\geq 1,
			\\
			2N_t(0) & \text{if} \, j=0,
			\\
			2N_t(0) - \sum^{-1}_{\ell=j} (1-2\eta(t,\ell)), & \text{if} \, j\leq -1.
		\end{array}
	\right.
\end{align*}
We define
\begin{align*}
	\cRexcl(s,t;j,k)
	=
	\CRvv\left(
		h(s,j), \, h(t,k)
	\right).
\end{align*}
\begin{theorem}\label{thm:main-TASEP}
	For all $j,k\in\ze$, we have
	\begin{equation*}
		\lim_{s\to\infty}
		\cRexcl(s,as;j,k)
		=
		\rhoKPZ(a).
	\end{equation*}
\end{theorem}
This proves the conjectured Formula 2.6 from \cite{FS16}.

We also consider a last-passage percolation model that can be mapped to TASEP. Let $\{W(i,j):\, i,j \geq 0\}$ be a collection of independent random variables such that
\begin{align*}
	W(i,j) &\sim \mathrm{Exp}(1), \quad i,j \ge 1,
	\\
	W(i,0) &\sim \mathrm{Exp}(1/2), \quad i \geq 1,
	\\
	W(0,j) &\sim \mathrm{Exp}(1/2), \quad j \geq 1
    \\
	W(0,0) &=0 .
\end{align*}
Let $m,n \geq 0$ and set $\cS(m,n)$ to be the collection of up-right paths $S$ with $S(0)=(0,0)$ and $S(m+n)=(m,n)$. The passage time of such $S$ is defined as
\begin{align*}
	T(S) = \sum_{k=1}^{n+m} W(S(k)).
\end{align*}
Finally, we define
\begin{align*}
	L(m,n) = \max\{T(S):\, S\in\cS(m,n)\}.
\end{align*}
Let
\begin{align*}
	\cRLPP(m_1,m_2;n_1,n_2)
	=
	\CRvv\left(
		L(m_1,n_1),\, L(m_2,n_2)
	\right).
\end{align*}
\begin{theorem}\label{thm:main-LPP}
	Let $a\ge 1$. Then,
\begin{align*}
		\lim_{n\to\infty}
		\cRLPP(n,\lfloor an \rfloor;n,\lfloor an \rfloor)
		=
		\rho^{\text{KPZ}}(a).
	\end{align*}
\end{theorem}
Both models satisfy the covariance-to-variance reduction which is a consequence of their space-time shift invariance by Lemma \ref{lem:stat-CVTV}.

\subsection{Semi-discrete directed polymers in a Brownian environment}\label{sec:results-semi-discrete}



\subsubsection{The partition function}\label{sec:results-semi-discrete-pf}


We introduce the model of semi-discrete directed polymers in a Brownian random environment from \cite{OY}.

We start defining the point-to-point partition function. Let $\{B_i(\cdot): i\geq 1\}$ be a collection of two-sided standard one-dimensional Brownian motions. For $0\le m< n$ and $s,t\in\re$ so that $s<t$ we define
\begin{align*}
	&
	\Delta(s,m;t,n):=\{(s_{m+1},\cdots,s_{n-1}):\, s<s_{m+1}<\dotsm<s_{n-1}<t \},\\
	&
	\Delta(t,n):=\{(s_0,\cdots,s_{n-1}):\, -\infty<s_{0}<\dotsm<s_{n-1}<t \}.
\end{align*}
For $\beta>0$, $1\le m+1 < n $ and $0\le s<t$, we define the point-to-point partition function
\begin{align}\label{eq:pf-ptp}
	Z^{\beta}&(s,m;t,n)\\
	&=
	\nonumber
\int_{\Delta(s,m;t,n)}
	\exp\bigl[ \beta \left\{B_{m+1}(s,s_{m+1}) + B_{m+2}(s_{m+1},s_{m+2})+\dotsm + B_n(s_{n-1},t)\right\}\bigr] \,ds_{m+1,n-1},
\end{align}
where we write $ds_{m+1,n-1}=ds_{m+1} \cdots ds_{n-1}$ and $B_k(u,v):=B_k(v)-B_k(u)$.

Fix a new two-sided standard one-dimensional Brownian motions $B_0$, independent of $\{B_i(\cdot): i\geq 1\}$. The stationary partition function $Z^{\beta,\theta}(t,n)$ is then defined as
\begin{align*}
	Z^{\theta,\beta}(t,n)
	&=
	\int_{\Delta(t,n)}
	\exp\bigl[ -\beta B_0(s_0) + \theta s_0  +  \beta \left\{B_1(s_0,s_1) +\dotsm + B_n(s_{n-1},t)\right\}\bigr] \,ds_{0,n-1},
\end{align*}
where we wrote $ds_{0,n-1}=ds_0 \cdots ds_{n-1}$.
Note that we have the representation
\begin{align*}
	Z^{\beta,\theta}(t,n)
	=
	\int^t_{-\infty} e^{-\beta B_0(s_0) + \theta s_0} Z^{\beta}(s_0,0;t,n)\, ds_0.
\end{align*}
The term stationary refers to a specific structure highlighted in Section~\ref{sec:proof-polymers-stationarity}.


We now state the covariance-to-variance reduction for the directed polymer model.
As $Z^{\theta,\beta}(t,n)$ is almost surely positive, we can define $H^{\theta,\beta}(t,n)=\log Z^{\theta,\beta}(t,n)$.
We have the following.
\begin{lemma}\label{thm:CVTV-polymers}
The logarithm of the stationary partition function $H^{\theta,\beta}$ follows the covariance-to-variance reduction.
\end{lemma}


\subsubsection{The intermediate disorder regime}\label{sec:results-semi-discrete-idr}


The relation between the polymer model and the stochastic heat equation is given by the following scaling limits known as intermediate disorder regime.
We define the rescaled stationary partition function as
\begin{eqnarray*}
	\she_n^{\textsc{st}}(t,x):=e^{-\sqrt{n}(\sqrt{n}t-x)-\tfrac{1}{2}(\sqrt{n}t-x)-\frac12 n \log n }Z^{\beta_n,\theta_n}(t n-x\sqrt{n},\lfloor t n \rfloor ),
\end{eqnarray*}
where $\beta_n = n^{-1/4}$ and $\theta_n = 1+\tfrac{\beta^2_n}{2}$ and here and in what follows, we systematically omit the integer part from the notation. %

%
Let $\kpz_n = \log \she_n^{\textsc{st}}$. This is known as the intermediate disorder scaling and from \cite{JM}, we know that $\kpz_n \Rightarrow \kpz$ as $n\to\infty$ in the locally uniform topology.
The first result of this type was on the Cole-Hopf level \cite{AKQ}. The result in \cite{JM} shows convergence of the discrete gradients of $\kpz_n$ to the energy solution to the stochastic Burgers equation i.e.\ the space-derivative of the KPZ equation. The convergence takes places in the space of distributions. However, the arguments of \cite[Section 6.4]{GJ-arma} can be easily adapted to transfer the result on the gradient convergence to the one of the KPZ to a convergence of the whole path $\kpz_n$ to KPZ. We include some heuristics on the intermediate disorder regime in Appendix \ref{app:idr} aimed at readers who are not familiar with this setting.

We can now state our Theorem on aging for directed polymers in the intermediate disorder regime. Let
\begin{align*}
	\cR_n(s,t;x,y)
	=
	{\CRvv\left( \kpz_n(s,x),\kpz_n(t,y)\right)}.
\end{align*}
\begin{theorem}\label{thm:main-polymers}
	Let $0\leq s < t$ and $x,y\in \re$. Then,
	\begin{align*}
		\lim_{n\to\infty} \cR_n(s,t;x,y) = \cR^{\text{KPZ}_{EQ}}(s,t;x,y).
	\end{align*}
\end{theorem}
Combining the above with Theorem~\ref{thm:main-SHE}, we get the statement
\begin{align*}
	\lim_{ s\to\infty}\lim_{n\to\infty} \cR_n(s,as;x,y)
	=
	\frac{1+a^{2/3}-(a-1)^{2/3}}{2a^{1/3}}.
\end{align*}

The key element of the proof of Theorem \ref{thm:main-polymers} is the convergence of the moments of $\kpz_n$ which is of independent interest and constitutes the most technical part of this work. The proof uses concentration inequalities based on Malliavin calculus to obtain uniform estimates on the lower tails of $\she_n$.
\begin{theorem}\label{thm:convergence-moments-polymers}
	For all $p>0$, $t\geq 0$ and $x\in\re$, we have
	\begin{align*}
		\lim_{n\to\infty}
		\esp\left[
			\kpz_n(t,x)^p
		\right]
		=
		\esp\left[
			\kpz(t,x)^p
		\right].
	\end{align*}
\end{theorem}
Note that, by Brownian scaling, it holds that
\begin{align*}
	Z^{\beta,\para}(t,n) = \beta^{-2n} Z^{1,\beta^{-2}\theta}(\beta^2 t, n).
\end{align*}
As a corollary of Lemma~\ref{thm:CVTV-polymers} and the estimates of Section~\ref{sec:proof-polymers-UI}, we obtain the following aging regime for polymers at a fixed temperature. Note that the parameter $\theta$ is still scaled with the size of the system.
Let
\begin{align*}
	R^{\theta,\beta}(s,t;m,n)
	&=
	{\CRvv(H^{\theta,\beta}(t,n),H^{\theta,\beta}(s,m))}.
\end{align*}
\begin{corollary}
	Let $\nu_n = \beta_n^{-2}\theta_n = \sqrt{n}+\frac{1}{2}$. Then,
	\begin{align*}
		\lim_{n\to\infty}
		R^{\nu_n,1}(s\sqrt{n}-x,t\sqrt{n}-y; sn, tn)
		=
		\cR^{\text{KPZ}_{EQ}}(s,t;x,y).
	\end{align*}
\end{corollary}
See Subsection~\ref{sec:results-open-questions-polymers-fixed-beta} for a discussion of the expected aging behavior for fixed $\beta$ and $\theta$.

\section{The Edwards-Wilkinson universality class}


\subsection{The Edwards-Wilkinson model}

As a comparison, we consider the Edwards-Wilkinson equation, that is, the stochastic heat equation with additive noise. As in the case of the KPZ equation, here the two-sided Brownian motion is stationary as well. More precisely, we consider the mild solution to the stochastic partial differential equation
\begin{align*}	
    \partial_t \,\ou(t,x) =&\;  \frac12 \partial_x^2 \,\ou(t,x) +  {\wh},\\
	\ou(0,x) =&\;  \cB(x),
\end{align*}
where $\cB(x)$ is a two-sided Brownian motion and ${\wh}$ is a space-time white noise.
In this case, the solution is given explicitly by convolution of the noise with the heat kernel (see \eqref{eq:mild-EW}).
We define
\begin{align*}
	\cR^{\text{EW}}(t_1,t_2;x_1,x_2)
	&=
	\CRvv\, (\ou(t_1,x_1),\ou(t_2,x_2)).
\end{align*}
\begin{theorem}\label{thm:main-add-SHE}
	For  all $a\ge 1$ and $x,y\in \re$, we have
	\begin{eqnarray}\label{eq:add-SHE-limit}
		\lim_{s\to\infty} \cR^{\text{EW}}(s,as; x,y) = \frac{1+a^{1/2}- (a-1)^{1/2}}{2 a^{1/4}}.
	\end{eqnarray}
\end{theorem}
Once again, the proof is based on the following.
\begin{lemma}\label{thm:CVTV-EW}
The mild solution $\ou$ to the EW equation follows the covariance-to-variance reduction.
\end{lemma}
Note that, by Brownian scaling, we have the $4:2:1$ identity $\ou(st,\sqrt{s}x)\overset{\mathrm{law}}{=}\sqrt{s}\ou(t,x)$ for all $s>0$. Together with the above covariance-to-variance reduction, we obtain the scaling relation
\begin{eqnarray}\label{eq:add-SHE-general-endpoints}
		\cR^{\text{EW}}(sa,sb;x\sqrt{s},y\sqrt{s}) = \cR^{\text{EW}}(a,b;x,y),
\end{eqnarray}
for all $s>0$, see a similar relation for $\cR^{\text{KPZ}}$ in \eqref{eq:remark3.2}. In particular,
\begin{eqnarray}\label{eq:add-SHE-exact-corr}
	\cR^{\text{EW}}(s,as; 0,0) = \frac{1+a^{1/2}- (a-1)^{1/2}}{2 a^{1/4}},
\end{eqnarray}
   for all $a\ge 1$.
The correlation $\cR^{{\text EW}}$ can, in fact, be made rather explicit even in the general case. In the next proposition, we obtain an explicit formula for the variance of the Edwards-Wilkinson model which, together with the covariance-to-variance reduction, yields an expression for $\cR^{{\text EW}}$. This is stated and proved in Proposition \ref{thm:variance-EW} is a slightly more general setting.
\begin{proposition}\label{thm:EW-exact-variance}
	Let $\sigma^2(x)=E_x[|B(1)|]$ where, under $P_x$, $B$ is a one-dimensional standard Brownian motion with $B(0)=x$. Then,
	\begin{eqnarray*}
		\Vvv \, \ou(1,x) = \sigma^2(x).
	\end{eqnarray*}
	As a consequence,
	\begin{eqnarray*}
		\cR^{{\text EW}}(a,b;x,y)
		&=&
		\frac{
			a^{1/2}\sigma^2(ya^{-1/2}) + b^{1/2}\sigma^2(yb^{-1/2}) - (b-a)^{1/2}\sigma^2((y-x)(b-a)^{-1/2})
		}
		{
			2(ab)^{1/4}
			\sqrt{ \sigma^2(xa^{-1/2}) \sigma^2(yb^{-1/2})}
		}
	\end{eqnarray*}
\end{proposition}

\begin{rk}
	We easily obtain the following asymptotics for $\rhoEW$:
	\begin{align*}
		\rhoEW(a)
		&\approx
		\frac{1}{2 a^{1/4}},
		\quad
		a\to\infty,
		\\
		1-\rhoEW(a)
		&\approx
		\mfrac12 (a-1)^{1/2},
		\quad
		a\to 1^+.
	\end{align*}
\end{rk}


\subsection{One-dimensional Ginzburg-Landau $\nabla$ interface model}\label{sec:results-gradient}

Let $V:\re\to[0,\infty)$ be a twice differentiable, symmetric and convex function such that
\begin{align*}
	0<c_1 \leq V''(x) \leq c_2 <\infty,\quad \forall \, x\in\re,
\end{align*}
for some constants $c_1,c_2 \in \re$. We consider $\{B_j:\, j\in\ze\}$ a family of independent standard one-dimensional Brownian motions and we let $\{u_j;\, j\in\ze\}$ be the solution to the system of coupled stochastic differential equations
\begin{align*}
	\dd u_j = {\frac12}\left( V'(\nabla u_{j-1})-V'(\nabla u_j) \right)\, \dd t +  \dd B_j,
\end{align*}
where $\nabla u_j = u_{j+1}-u_j$. The existence of the dynamics on suitable weighted spaces can be found for instance in \cite{Zhu}. It is a well-known fact that the gradient process $\{ \nabla u_j:\, j\in \ze\}$ admits a family of product invariant distributions given by
\begin{align*}
	\mu_{\lambda}(d\nabla u_j)
	=
	\frac{1}{C(\lambda)} e^{\lambda \grad u_j - V(\grad u_j)} d\grad u_j,
\end{align*}
where
\begin{align*}
	C(\lambda) = \int e^{\lambda \grad u_j - V(\grad u_j)} d\grad u_j,
\end{align*}
which is finite by the convexity of $V$.

We define
\begin{align*}
	\cR_{V}^{\text{GL}}(s,t;j,k)
	=
	{\CRvv(u_j(s),u_k(t))}.
\end{align*}
\begin{theorem}\label{thm:aging-GL}
	Consider the process $\{u_j(t),\, t\geq 0,\, j\in\ze\}$ with initial distribution $u_0=0$ and $\nabla u_j(0)\sim\mu_0^{\ze}$.
	Then for all $a\ge 1$ and $x,y\in \re$, we have
	\begin{align*}
		\lim_{s\to\infty} \cR_V^{\text{GL}}(s,as; x\sqrt{s},y\sqrt{s})
		=
		\cR^{\text{EW}}(1,a;x,y).
	\end{align*}
	In particular, for all $a\ge 1$, it holds that
	\begin{align*}
		\lim_{s\to\infty} \cR_V^{\text{GL}}(s,as; 0,0)
		=
		\frac{1+a^{1/2}- (a-1)^{1/2}}{2 a^{1/4}}.
	\end{align*}
\end{theorem}

\begin{rk}
For the Landau-Ginzburg model with i.i.d.\ initial distribution a
different and less precise asymptotic has been derived in \cite[Theorem 1.1]{DD}:
\[
\mathcal R_V^{\mathrm{GL},\mathrm{i.i.d}}(s,as;0,0)\asymp\frac{(a+1)^{1/2}-(a-1)^{1/2}}{2^{1/2}a^{1/4}}<\rho^{EW}(a)
\quad\text{for all }a>1.
\]
This is due to the effect of the initial distribution which is less
correlated than the stationary distribution which is of i.i.d.\ increments.
\end{rk}

An important element in the proof of Theorem \ref{thm:aging-GL} is a representation of the variance of the model in terms of random walks, in the spirit of Proposition \ref{thm:EW-exact-variance}.

\section{Open problems}\label{sec:results-open-questions}


Our path to aging consists in obtaining a covariance-to-variance reduction and taking the limit of the variances. This second step is the more involved as it requires to obtain highly non-trivial tail bounds to insure enough uniform integrability.

\subsection{Semi-discrete polymers at fixed temperature}\label{sec:results-open-questions-polymers-fixed-beta}


Our result for directed polymers holds in the intermediate disorder regime, where the parameters of the model depend on the size of the system. The model satisfies a covariance-to-variance reduction for fixed parameters and it is known that there exists two suitable constants $c_1$ and $c_2$ such that
\begin{align*}
	\frac{\log Z^{\beta,\theta}(n,n)-c_2n}{c_2 n^{1/3}}
	\Rightarrow
	X,
\end{align*}
where $X$ follows the Baik-Rains distribution \cite{BCF}. It is therefore expected that the correlations of the model converge to $\rhoKPZ$. The missing point is the convergence of the moments. For some results in this direction, see the recent work \cite{NS1}.


\subsection{The parabolic Anderson model}\label{sec:results-open-questions-PAM}


One of the original motivations for this work was to show aging for the Parabolic Anderson Model (PAM) in a Brownian potential, a problem that was left open in \cite{DD}. The PAM is an example of a discrete stochastic heat equation and can be defined as the solution to the infinite system of coupled stochastic differential equations
\begin{align*}
	dZ_j = \frac12 \Delta Z_j \dd t+ \beta Z_j dB_j, \quad j\in\ze,
\end{align*}
where $\Delta$ is the discrete Laplacian and $(B_j)_j$ is an i.i.d.\ family of one-dimensional Brownian motions. The model can be interpreted as the partition function of a continuous-time simple symmetric random walk in an i.i.d.\ Brownian potential, i.e.\ a directed polymer model. As such, it should converge to the stochastic heat equation in the intermediate disorder regime much in the spirit of the O'Connell-Yor model. However, no stationarity is known which leaves the covariance-to-variance reduction out of reach. 


\subsection{Further stationary directed polymer models}\label{sec:results-open-questions-polymers-other}

There exist other stationary polymer and last-passage percolation models to which the techniques of this work could be applied. We refer the reader to the recent work \cite{NS2} where concentration bounds for several of these models are obtained.


\subsection{Weakly asymmetric models}\label{sec:results-open-questions-weakly-asymmetric}


We were able to show aging for the height function of TASEP. However, the asymmetric exclusion process (ASEP) is also in the KPZ universality class and possesses product Bernoulli space-time invariant measures. The covariance-to-variance reduction is hence available. The missing point to show aging is the convergence of the moments. The asymmetry can be tuned to get convergence to the Cole-Hopf solution to the KPZ equation \cite{BG} and it is then reasonable to expect a result in the spirit of Theorem \ref{thm:main-polymers} in this regime.

On the other hand, a stationary and weakly asymmetric version of the Ginzburg-Landau $\nabla$-interface models considered here is known to rescale to the KPZ equation \cite{DGP}. Once again, an analogue of Theorem \ref{thm:main-polymers} is expected in this case.


\subsection{KPZ fixed point regime for discrete models}\label{sec:results-open-questions-KPZ-fixed-point}


As noticed above, Theorem \ref{thm:main-SHE} still holds if we consider end-points $x=x_s$ and $y=y_s$ with $x_s,y_s=o(s^{2/3})$. When these end-points are of order $s^{2/3}$, we expect that
\begin{align*}
	\lim_{s\to\infty}
	\cRKPZeq(st_1,st_2; s^{2/3}x_1,s^{2/3}x_2)
	=
	\cRKPZ(t_1,t_2; x_1,x_2),
\end{align*}
as, in this regime, $\she$ should rescale to the KPZ fixed point after proper centering and normalization. The proof of this fact is currently well far behind the reach of this work. Similar considerations should hold for other models in the KPZ universality class, including TASEP, LPP and weakly asymmetric gradient models.


\subsection{Non-stationary models}

Our approach relies heavily on the stationary
structure of the models we consider. However, aging should still hold for
these models with different initial conditions.
See for example in \cite{DD} for various interacting diffusion and \cite{CGH20} for the
KPZ in the narrow edge regime.


\subsection{Higher dimensions}\label{sec:results-open-questions-higher-dimension}


Finally, we address the question of aging for higher dimensional models of directed polymer type. The parabolic Anderson model described above can be defined on any dimension. It is expected that aging still holds in dimension two, in logarithm time scale, cf.\ \cite[Theorem 1.1]{DD}. In larger dimensions, the model displays a phase transition as the parameter $\beta$ varies. For small $\beta$, the model is in a regime of weak disorder and admits an equilibrium measure which can be constructed by means of suitable limits of the solutions using standard techniques from the theory of directed polymers in random environment. Hence, no aging should occur. On the other end, large values of $\beta$ lead to a completely different strong disorder regime where some kind of aging should hold.


\section{Proofs for the KPZ fixed point}


We follow \cite{MQR}. Let $h$ be the height function of the stationary TASEP and let
\begin{align*}
	\kpzfp^{\varepsilon}(t,x)
	=
	\varepsilon^{1/2}
	\left(
		h(2\varepsilon^{-2/3}t, 2\varepsilon^{-1}x)+\varepsilon^{-3/2}t
	\right).
\end{align*}
Then, the KPZ fixed point $\kpzfp$ is defined as the limit in law
\begin{align*}
	\kpzfp(t,x)
	=
	\lim_{\varepsilon\to0}
	\kpzfp^{\varepsilon}(t,x).
\end{align*}
The convergence above is as processes in a suitable topology proved in \cite{MQR}, and particularly in the sense of the finite dimensional distributions (see Theorem 3.8 there). From the same article, we have the scaling identity $\kpzfp(st,s^{3/2}x)\overset{\mathrm{law}}{=}s^{1/3}\kpzfp(t,x)$ and stationarity: $x\mapsto \kpzfp(t,x+y)-\kpzfp(t,y)$ is a two-sided Brownian motion for all $t\geq 0$. The covariance-to-variance reduction follows from space-time stationarity by Lemma \ref{lem:stat-CVTV}.
\begin{proposition}\label{prop:st_stat_KPZ_FP}
	For all $t\geq 0$ and $x\in \re$, we have
	\begin{align*}
		\kpzfp(t+\cdot,x+\cdot)-\kpzfp(t,x)
		\overset{\mathrm{law}}{=}
		\kpzfp(\cdot,\cdot).
	\end{align*}
	In other words, the process $\kpzfp$ is space-time stationary.
\end{proposition}
This is a corollary of the space-time invariance for $h$ itself:
\begin{lemma}\label{thm:TASEP-stationarity}
	For all $t\geq 0$ and $k\in\ze$, we have
	\begin{align*}
		h(t+\cdot,k+\cdot)-h(t,k)
		\overset{\mathrm{law}}{=}
		h(\cdot,\cdot).
	\end{align*}
\end{lemma}
\begin{proof}
	We will use two different representations of the height function. Recall that $N_t(j)$ denotes the flux of particles passing through $j$ integrated over the time interval $[0,t]$. Then, it holds that
	\begin{align*}
		h(t,j)=h(0,j)+2N_t(j).
	\end{align*}
	On the other hand, we also have 
	\begin{align*}
		h(t,j+k) = h(t,j) + \sum^{j+k-1}_{l=j} \hat\eta(t,l),
	\end{align*}
	where $\hat{\eta}(t,l):=1-2\eta(t,l)$.
	For $t=0$, the above reduces to
	\begin{align*}
		h(0,j) =  \sum^{j-1}_{l=0} \hat\eta(0,l).
	\end{align*}
	Note that the process $\hat\eta$ is space-time invariant. Now
	\begin{align*}
		h(t+s,k+j)
		&=
		h(0,k+j)+2N_{t+s}(k+j)
		\\
		&=
		h(0,k+j)+2N_{t}(k+j)+ 2\left( N_{t+s}(k+j) - N_t(k+j) \right)
		\\
		&=
		h(t,k+j) + 2N_{t,t+s}(k+j),
	\end{align*}
	where $N_{t,t+s}(k+j)$ is the flux of particles through $j+k$ integrated over the time interval $[t,t+s]$. We now use the second representation as a coupled system, that is for every $t\ge0$ and $j\in\bZ$ we have the equality of processes:
	\begin{align*}
		\big( h(t+s,k+j) - h(t,k) \big)_{s\ge0,j\in\bZ}
		&=
		\big( h(t,k+j) - h(t,k)+ 2N_{t,t+s}(k+j) \big)_{s\ge0,j\in\bZ}
		\\
		&=
		\big(\sum^{k+j-1}_{l=k} \hat\eta(t,l) + 2N_{t,t+s}(k+j))\big)_{s\ge0,j\in\bZ}
	\end{align*}
	Using the space-time invariance of $\hat\eta$ it follows that the later is equal in law to
	\begin{align*}
		\big(\sum^{j-1}_{l=0} \hat\eta(0,l) + 2N_{s}(j) \big)_{s\ge0,j\in\bZ}
		=
		\big( h(s,j) \big)_{s\ge0,j\in\bZ}. 
	\end{align*}
\end{proof}

\begin{proof}[Proof of Lemma \ref{thm:CVTV-KPZ-FP}]
	By Lemma \ref{lem:stat-CVTV}, the covariance-to-variance reduction follows from the space-time stationarity.
	It is enough to prove that $\kpzfp(t,x)$ has a finite second moment for all $t\geq 0$ and $x\in\re$. By the scaling identity, it is enough to consider $t=1$. From \cite[Theorem 1]{BFP}, for $x\in\re$, there exists a distribution $F_x$ with finite moments such that
	\begin{eqnarray}\label{eq:TASEP-convergence-moments}
		\lim_{\varepsilon\to0}
		\esp\left[
			\kpzfp^{\varepsilon}(1,x)^n
		\right]
		=
		\int y^n dF_x(y),
\end{eqnarray}
for all $n\geq 0$. By Fatou's lemma, we then have that $\esp[\kpzfp(1,x)^2]<\infty$.
\end{proof}

\begin{proof}[Proof of Theorem \ref{thm:main-KPZ-FP}]
	For $x=y=0$, the scaling identity satisfied by $\kpzfp$ immediately yields $\cRKPZ(s,as;0,0)=\rhoKPZ(a)$.
	To include general end-points, we rewrite the covariance-to-variance reduction as
	\begin{eqnarray*}
		\cRKPZ(s,as;x,y)
		=
		\frac{
			s^{-2/3} \Vvv\, \kpzfp(s,x)
			+
			s^{-2/3} \Vvv\,\kpzfp(as,y)
			-
			s^{-2/3} \Vvv\, \kpzfp((a-1)s,y-x)
		}
		{
			2 \sqrt{
				s^{-2/3} \Vvv\, \kpzfp(s,x)
				\,
				s^{-2/3} \Vvv\, \kpzfp(as,y)
			}
		}.
	\end{eqnarray*}
	Let $B(t,x)=\kpzfp(t,x)-\kpzfp(t,0)$ and recall that for each fixed $t$, $x\mapsto B(t,x)$ is a two-sided Brownian motion. Now,
	\begin{eqnarray*}
		\Vvv \,\kpzfp(s,x)
		=
		\Vvv \,\kpzfp(s,0)
		+
		\Vvv \, B(s,x)
		+
		2
		\,
		\Cvv\left( \kpzfp(s,0),\, B(s,x)\right).
	\end{eqnarray*}
	For fixed $x$, it holds that $\lim_{s\to \infty} s^{-2/3} \, \Vvv\, B(s,x) = \lim_{s\to \infty} s^{-2/3} |x| = 0$ and by Cauchy-Schwartz' inequality
	\begin{eqnarray*}
		\nonumber
		\frac{| \Cvv\left( \kpzfp(s,0),\, B(s,x)\right)|}{s^{2/3}}
		&\leq&
		\left( \frac{\Vvv\, \kpzfp(s,0)}{s^{2/3}} \right)^{1/2}
		\left( \frac{\Vvv\,B(s,x)}{s^{2/3}} \right)^{1/2}
		\\		
		&=&
		\left( \Vvv\, \kpzfp(1,0) \right)^{1/2}
		\left( \frac{\Vvv\,B(s,x)}{s^{2/3}} \right)^{1/2}
	\end{eqnarray*}
	which converges to $0$ as $s$ goes to infinity. Hence, using scaling one more time,
	\begin{align*}
		&\lim_{s\to\infty}
		\cRKPZ(s,as;x,y)
		\\
		&=
		\lim_{s\to\infty}
		\frac{
			s^{-2/3} \Vvv\, \kpzfp(s,0)
			+
			s^{-2/3} \Vvv\,\kpzfp(as,0)
			-
			s^{-2/3} \Vvv\, \kpzfp((a-1)s,0)
		}
		{
			2 \sqrt{
				s^{-2/3} \Vvv\, \kpzfp(s,0)
				\,
				s^{-2/3} \Vvv\, \kpzfp(as,0)
			}
		}
		\\
		&=
		\lim_{s\to\infty}
		\cRKPZ(s,as;0,0)
		=
		\rhoKPZ(a).
	\end{align*}
\end{proof}


\section{Proofs for the stochastic heat equation}\label{sec:proof-SHE}


By a mild solution to the stationary stochastic heat equation, we mean a progressively measurable process solving the integral equation
\begin{align*}
	\she(t,x)
	=
	\int_{\re} p(t,x-y) e^{\cB(y)} dy
	+
	\int^t_0 \int_{\re} p(t-s,x-y) \she(s,y) \wh(ds dy).
\end{align*}
We refer the reader to the early reference \cite{BG} for a proof of existence and uniqueness among other results.


\subsection{Space-time stationarity}\label{sec:proof-SHE-stationarity}


It is well known that, for each fixed $t\geq 0$, the process
\begin{align*}
	x \mapsto \log \she(t,x)
\end{align*}
has Brownian increments \cite{BG, FQ}.

By Lemma \ref{lem:stat-CVTV} the covariance-to-variance reduction follows from the following Proposition.
\begin{proposition}\label{thm:eq-in-law-SHE}
	Let $s\geq 0$ and $y\in\re$. Then, we have the identity
	\begin{align*}
		\frac{\she(\cdot,\cdot)}{\she(s,y)}
		\overset{\mathrm{law}}{=}
		\she(\cdot-s,\cdot-y).
	\end{align*}
	In other words, the process $\kpz = \log \she$ is space-time stationary.
\end{proposition}
\begin{proof}
	For $u\geq 0$ and $b\in\re$, we define a new white noise $\wh_{u,b}$ acting on $L^2([0,\infty) \times \re)$ as
	\begin{align*}
		\int^\infty_0 \int_{\re} f(r,z) \wh_{u,b}(dr dz)
		=
		\int^\infty_0 \int_{\re} f(r+u,z+b) \wh(dr dz).
	\end{align*}
	Note that $\wh_{u,b}$ is independent of the restriction of $\wh$ to the strip $[0,u]\times \re$. Using the flow property of the mild solution,
	\begin{align*}
		\she(t,x)
		&=
		\int_{\mathbb{R}} p (t - s , x - z) \she (s, z) d z
		\\
		&
		+
		 \int^t_s
  		\int_{\mathbb{R}} p (t - r, x - z) \she (r , z) \wh (dr d z)
  		\\
  		&=
		\int_{\mathbb{R}} p (t - s, x - y - z) \she (s, y+z ) d z
		\\
		&
		+
		 \int_0^{t - s}
  		\int_{\mathbb{R}} p (t - s - r, x - y - z) \she (s + r, y + z) \wh_{s, y} (d r d z).
	\end{align*}
	Now, using the independence of $\she(s,y)$ and $\wh_{s,0}$ to pull $\she(s,y)$ inside the stochastic integral,
	\begin{align*}
		\frac{\she(t,x)}{\she(s,y)}
  		&=
		\int_{\mathbb{R}} p (t - s, x - y - z) \frac{\she (s, y+z)}{\she(s,y)} d z
		\\
		&
		+
		 \int_0^{t - s}
  		\int_{\mathbb{R}} p (t - s - r, x - z) \frac{\she (s + r, y+z)}{\she(s,y)} \wh_{s, y} (d r d z).
	\end{align*}
	This shows that $\displaystyle [s,\infty) \times \re \ni (t,x) \mapsto \frac{\she(t,x)}{\she(s,y)}$ is a mild solution to the stochastic heat equation with initial condition $\displaystyle \frac{\she(s,\cdot)}{\she(s,y)}$ which is independent of the noise $\wh_{s, y}$ for all $t>s$ and, by stationarity, is distributed as the exponential of a two-sided Brownian motion shifted by $y$. This proves the claim.
\end{proof}


\subsection{Proof of Theorem~\ref{thm:main-SHE}}\label{sec:proof-SHE-main}


We prove Theorem~\ref{thm:main-SHE}. First, we quote a few results from \cite{BCFV} and \cite{CG}. Let
\begin{eqnarray*}
	h(t,x)
	&=&
	\frac{\kpz(2t,x)+\frac{t}{12}-\frac23 \log 2t}{t^{1/3}}.
\end{eqnarray*}
Recall the definition of the Baik-Rains distribution from \cite[Definition 2.16]{BCFV}. From \cite[Theorem 2.17]{BCFV}, we know that
\begin{eqnarray}\label{eq:BR-SHE}
	h(t,0) \Rightarrow X \quad \text{as} \quad t\to\infty,
\end{eqnarray}
where $X$ follows the Baik-Rains distribution.
From \cite[Corollary 1.14]{CG}, we also have convergence of the moments: for all $p>0$,
	\begin{eqnarray}\label{eq:BR-SHE-moments}
		\lim_{t\to\infty}
		\esp[h(t,0)^p]
		=
		\esp[X^p].
	\end{eqnarray}
The key to this result is a combination of two tail bounds for $h(t,0)$ ensuring that the family $\{h(t,0)^p:\, t\geq 0\}$ is uniformly integrable (see 	\cite[Theorem 1.12 and 1.13]{CG}).

We can now prove Theorem~\ref{thm:main-SHE}.
\begin{proof}[Proof of Theorem~\ref{thm:main-SHE} ]
	
	First, assume $x=y=0$. Let $\sigma^2(t)=t^{-2/3} \Vvv\, \kpz(t,0)$. We have
	\begin{eqnarray*}
		\nonumber
		\cRKPZeq(s,as; 0,0)
		&=&
		\frac{
			(as)^{2/3} \sigma^2(as)
			+
			s^{2/3} \sigma^2(s)
			-
			s^{2/3}(a-1)^{2/3}\sigma^2((a-1)s)
		}{
			2s^{2/3}a^{1/3}
			\sqrt{ \sigma^2(s)\, \sigma^2(as)}
		}.
	\end{eqnarray*}
	It is then enough to show convergence of $\sigma^2(t)$ as $t\to\infty$.
	From~\eqref{eq:BR-SHE}, we have the convergence in law
	\begin{eqnarray*}
		\frac{\kpz(t,0)+\frac{t}{24}-\frac23 \log t}{t^{1/3}}
		\Rightarrow
		2^{1/3} X,
	\end{eqnarray*}
	where $X$ follows the Baik-Rains distribution. From~\eqref{eq:BR-SHE-moments}, we obtain
	\begin{eqnarray*}
		\lim_{t\to\infty}\frac{\Vvv \,\kpz(t,0)}{t^{2/3}}=4^{1/3}\, \Vvv\, X.
	\end{eqnarray*}
	We can then take limits on the right-hand-side of the covariance-to-variance reduction in Lemma~\ref{thm:CVTV-SHE}.
	The proof To handle general end-points, we can replicate the arguments in the proof of Theorem \ref{thm:main-KPZ-FP}.
\end{proof}


\section{Proofs for TASEP and LPP}

We already have all the elements to prove Theorem \ref{thm:main-TASEP}.
\begin{proof}[Proof of Theorem \ref{thm:main-TASEP}]
	We follow the scheme of proof of Theorem \ref{thm:main-KPZ-FP}.
	The covariance-to-variance reduction for $h$ follows from Lemma \ref{thm:TASEP-stationarity}. The convergence of the variances follows from \eqref{eq:TASEP-convergence-moments}.
\end{proof}
The proof for TASEP follows from stationarity and an exact mapping to TASEP. For the former:
\begin{lemma}\label{thm:LPP-stationarity}
	For all $m,n\geq 0$, we have
	\begin{align*}
		L(\cdot,\cdot)-L(m,n)
		\overset{\mathrm{law}}{=}
		L(\cdot-m,\cdot-n).
	\end{align*}
\end{lemma}
\begin{proof}
	Let $ m_1,n_1\ge 0$.
	We define new random variables $\{\overline{W}(i,j):\, i\geq n_1,\, j\geq m_1\}$ as
	\begin{align*}
		\overline{W}(i,j) &= W(i,j), \quad i>n_1\, \text{ and } \, j>m_1,
		\\
		\overline{W}(n_1,j) &= L(n_1,j)-L(n_1,j-1), \quad j>m_1,
		\\
		\overline{W}(i,m_1) &= L(j,m_1)-L(j-1,m_1), \quad i>n_1
        \\
        \overline{W}(n_1,m_1)&=0.
	\end{align*}
	It is known that $\{\overline{W}(i,j):\, i\geq n_1,\, j\geq m_1\}$ is a family of independent random variables such that
	\begin{align*}
		\overline{W}(i,j) &\sim \mathrm{Exp}(1) \text{ for } i>n_1,\, j>m_1,\\
		\overline{W}(i,j) &\sim \mathrm{Exp}(1/2) \text{ for } i=n_1,j>m_1  \text{ or } i>n_1,\, j=m_1. 
	\end{align*}
	For any $m_2\ge m_1$ and $n_2\ge n_1$ and an up-right path $S$ from $(m_1,n_1)$ to $(m_2,n_2)$, we define the passage time $\overline{T}(S)$ using the random variables $\overline{W}(i,j)$ and we let $\overline{L}(m_2,n_2)$ to be the maximum of $\overline{T}(S)$ over such paths. Then, it holds that $\big(\overline{L}(m_2,n_2)\big)_{m_2\ge m_1,n_2\ge n_1}$ has the same law as $\big(L(m_2-m_1, n_2-n_1)\big)_{m_2\ge m_1,n_2\ge n_1}$.
	Let $\mathcal{B}(m_1,n_1)=\{(i,j):\, i=m_1 \, \text{ and }\, n_1 \leq j \leq n_2 \, \text{ or } \,m_1 \leq i \leq m_2 \, \text{ and } \, j=m_1\}$.
	Let $S^*$ be the optimal path from $(0,0)$ to $(m_2,n_2)$. $S^*$ crosses $\mathcal{B}$ and exits it at a point $x^*$. Without loss of generality, assume that $x^*$ lies on the horizontal segment of $\mathcal{B}$. Let $x^*_+=x^*+(0,1)$. Then,
	\begin{align*}
		L(m_2,n_2)
		&=
		L(x^*) + L(x^*; m_2,n_2)
		\\
		&=
		L(m_1,n_1) + \left( L(x^*)-L(m_1,n_1) + L(x^*_+; m_2,n_2) \right)
		\\
		&\leq
		L(m_1,n_1) + \overline{L}(m_2,n_2).
	\end{align*}
	Now, there exists $y^{*} \in \mathcal{B}$ such that
	\begin{align*}
		\overline{L}(m_2,n_2)
		&=
		L(y^*)-L(m_1,n_1) + L(y^*_+;m_2,n_2),
	\end{align*}
	where we assumed, without loss of generality, that $y^*$ lies on the horizontal segment of $\mathcal{B}$ and we let $y^*_+=y^*+(0,1)$.
	Hence,
	\begin{align*}
		L(m_1,n_1)+\overline{L}(m_2,n_2)
		&=
		L(m_1,n_1) +L(y^*) -L(m_1,n_1)+ L(y^*_+;m_2,n_2)
		\\
		&=
		L(y^*)+L(y^*_+;m_2,n_2)
		\\
		&\leq
		L(m_2,n_2). \qedhere
	\end{align*}
To conclude, note that we in fact showed that $\big(\overline{L}(m_2,n_2)\big)_{m_2\ge m_1,n_2\ge n_1} = \big(L(m_2,n_2) -L(m_1,n_1)\big)_{m_2\ge m_1,n_2\ge n_1}$.
\end{proof}

\begin{proof}[Proof of Theorem \ref{thm:main-LPP}]
	We use the well-known identity
	\begin{eqnarray*}
		\p\left[
			L(n,n) \leq t
		\right]
		=
		\p\left[
			N_t(0) \geq n
		\right]
		=
		\p\left[
			h(t,0) \geq n/2
		\right].
	\end{eqnarray*}
	Then, the convergence of the moments of the rescaled height function yields convergence of the moments of
	\begin{align*}
		\frac{L(n,n)-c_1n}{c_2 n^{1/3}},
	\end{align*}
	where $c_1$ and $c_2$ are properly chosen constants. The proof then follows along the lines of the proof of Theorem \ref{thm:main-KPZ-FP}.
\end{proof}


\section{Proofs for the Edwards-Wilkinson equation}\label{sec:proof-SHE-additive}

Recall that a solution to the Edwards-Wilkinson equation (or stochastic heat equation with additive noise) corresponds to the mild solution to the equation
\begin{align*}	
    \partial_t \, \ou(t,x) =&\; \tfrac12 \partial_x^2 \, \ou(t,x) + \wh,
	\\
	\ou(0,x) =&\; \cB(x),
\end{align*}
The solution is explicitly given by
\begin{eqnarray}\label{eq:mild-EW}
	\ou(t,x)
	=
	\int_{\re}p(t,x-z) \cB(z)\, dz
	+
	\int^t_0 \int_{\re} p(t-r,x-z) \, dzdr,
\end{eqnarray}
where
\begin{align*}
	p(t,x) = \frac{1}{\sqrt{2\pi t}} e^{-\frac{x^2}{2t}}.
\end{align*}
As $\ou$ arises as the limit of discrete models which discrete gradients have product invariant distributions \cite{Zhu, GJS}, it follows that the two-sided Brownian motion initial condition is stationary.

We now sketch the proof of Theorem~\ref{thm:main-add-SHE}. The formula \eqref{eq:add-SHE-general-endpoints} is a consequence of the equality in law
\begin{align*}
	\ou(as,bs;x\sqrt{s},y\sqrt{s}) \overset{\mathrm{law}}{=} s^{1/4}\ou(a,b;x,y),
\end{align*}
which follows from Brownian scaling. The covariance-to-variance reduction which will again play a crucial role follows from Proposition~\ref{thm:EW-stationarity} below. Formula \eqref{eq:add-SHE-exact-corr} then follows from the covariance-to-variance reduction and the scaling relation above. Finally, the limit \eqref{eq:add-SHE-limit} follows from the arguments used at the end of the proof of Theorem~\ref{thm:main-SHE} once we note that
\begin{align*}
	\ou(t,x) = \ou(t,0) + \left( \ou(t,x) - \ou(t,0) \right),
\end{align*}
and $\Vvv \, \left( \ou(t,x) - \ou(t,0) \right) = |x|$ for all $t\geq 0$.

As noted above, to get the covariance-to-variance reduction, it is enough to show:
\begin{proposition}\label{thm:EW-stationarity}
	For all $s\geq 0$ and $y\in \re$, we have the identity
	\begin{align*}
		\ou(\cdot,\cdot)-\ou(s,x) \overset{\mathrm{law}}{=}  \ou(\cdot-s,\cdot-x).
	\end{align*}
\end{proposition}
\begin{proof}
	The proof is very similar to the multiplicative case. By the flow property of the mild solution,
	\begin{eqnarray*}
		\nonumber
		\ou(t,y)
		&=&
		\int_{\re} p(t-s,x-z) \, \ou(s,z)\, dz
		+
		\int^t_s \int_{\re} p(t-r,x-z) \wh(dr dz)
		\\
		&=&
		\int_{\re} p(t-s,x-y-z) \, \ou(s,y+z)\, dz
		+
		\int^{t-s}_0 \int_{\re} p(t-s-r,x-y-z) \wh_{s,y}(dr dz),
	\end{eqnarray*}
	where $\wh_{s,y}$ was defined in the proof of Proposition~\ref{thm:eq-in-law-SHE}. Hence,
	\begin{align*}
		\ou(t,y)-\ou(s,x)
		&=
		\int_{\re} p(t-s,x-y-z) \, \left( \ou(s,y+z)-\ou(s,y) \right) dz
		\\
		&
		+
		\int^{t-s}_0 \int_{\re} p(t-s-r,x-y-z) \wh_{s,y}(dr dz).
	\end{align*}
	To conclude, we just note that $z\mapsto \ou(s,y+z)-\ou(s,y)$ is a two-sided Brownian motion which is independent of $\wh_{s,y}$.
\end{proof}


The next result will be used in the proof of Theorem~\ref{thm:aging-GL} and might be of independent interest.
\begin{proposition}\label{thm:variance-EW}
	We have the formula
	\begin{eqnarray*}
		\Vvv \, \ou(t,x)
		=
		|x| + \int^t_0 p(s,x)\, ds
		=
		E_x[|B_t|],
	\end{eqnarray*}
	where, under $E_x$, $B$ is a Brownian motion with $B_0=x$.
\end{proposition}
\begin{rk}
	We can obtain slightly more explicit expressions for $\Vvv \, \ou(t,x)$. Note that we have the scaling
	\begin{align*}
		\Vvv \, \ou(t,x)
		=
		r(t,x)
		=
		t^{1/2}r(1,xt^{-1/2}),
	\end{align*}
	where
	\begin{align*}
		r(1,y)=|y|(2\Phi(|y|)-1)+2p(1,y),
		\qquad
		\Phi(y)=\int_{-\infty}^{|y|}p(1,z)dz.
	\end{align*}
\end{rk}

\begin{proof}[Proof of Proposition~\ref{thm:variance-EW}]
	We write $p_t(x)=p(t,x)$ to lighten the notation.
	From \eqref{eq:mild-EW} and the independence of $\cB$ and $\wh$, we have
	\begin{align*}
		\Vvv\,(\ou(t,x))
		=
		\int_{\re}\int_{\re}p_t(z-x)p_t(\bar z-x)g(z,\bar z)\,dz\,d\bar z
		+
		\frac12 \int_0^{2t}p_u(0)\,du,
	\end{align*}
	where
	\begin{align*}
		g(z,\bar z)
		=
		z\wedge\bar z {\bf 1}_{z,\bar z\ge 0}
		+
		(-z)\wedge(-\bar z) {\bf 1}_{z,\bar z\le 0}.
	\end{align*}
	We claim that
	\begin{eqnarray}\label{eq:derivative-of-the-variance}
		\partial_t \, \Vvv\,(\ou(t,x))=p_t(x).
	\end{eqnarray}
	Together with $\Vvv\,(\Cal \ou(0,x))=|x|$, this proves the first identity above.
	
	For fixed $x$, we set $f_t(z)=p_{t}(x-z)$. Then,
	\begin{align*}
		\partial_t \, \Vvv\,(\Cal Z(t,x))
		=
		2\int_{\re}\int_{\re} \partial_t f_t(z)f_t(\bar z)g(z,\bar z)\,dz\,d\bar z
		+
		f_{2t}(x).
	\end{align*}
	Using that $\partial_t f_t(z)=\frac12f''_t(z)$, it holds that
	\begin{align*}
		2\int_{\re} f_t(\bar z)\int_{\re}\partial_t f_t(z)g(z,\bar z)\,dz\,d\bar z
		=
		\int_{\re} f_t(\bar z)\,\Big(\int_{\re}f''_t(z)g(z,\bar z)\,dz\Big) \, d\bar z.
	\end{align*}
	By integration by parts,
	\begin{align*}
		\int_{\re}f''_t(\bar z)g(z,\bar z)\,dz
		=
		-\int_{\re}f'_t(z)\partial_z g(z,\bar z)\,dz,
	\end{align*}
	where
	\begin{align*}
		\partial_z g(z,\bar z)=-{\bf 1}_{\bar z\le z\le0}+{\bf 1}_{0\le z\le\bar z}.
	\end{align*}
	Thus
	\begin{align*}
		-\int_{\re}f'_t(z)\partial_z g(z,\bar z)\,dz
		&=
		\int_{\bar z}^0f'_t(z)\,dz1_{\bar z\le 0}-\int_0^{\bar z}f'_t(z)\,dz1_{\bar z\ge 0}\\
		&=
		\big(f_t(0)-f_t(\bar z)\big)1_{\bar z\le 0}-\big(f_t(\bar z)-f_t(0)\big)1_{\bar z\le 0}\\
		&=
		f_t(0)-f_t(\bar z),
	\end{align*}
	and we get
	\begin{align*}
		\int_{\re} f_t(\bar z)\,d\bar z\Big(\int_{\re}f''_t(z)g(z,\bar z)\,dz\Big)
		&=
		f_t(0)\int_{\re} f_t(\bar z)\,d\bar z-\int_{\re} f^2_t(\bar z)\,d\bar z=p_t(x)-p_{2t}(0),
	\end{align*}
	where we have used that the integral of $f_t$ is equals to one and the semigroup property.
	Putting things together and recalling that $f_{2t}(x)=p_{2t}(0)$, we obtain \eqref{eq:derivative-of-the-variance}.
	
	Next we claim that
	\begin{eqnarray}\label{eq:Brownian-identity}
		r(t,x):=E_x[|B(t)|] = |x|+\int_0^t p_s(x)\,ds,
	\end{eqnarray}
	where $B(t)$ is a Brownian motion.
	Note that
	\begin{align*}
		r(t,x)
		=
		\int_{\re}p_t(z)|z+x|\,dz
		=
		\int_{\re	}p_t(x-z)|z|\,dz=-\int_\infty^0f_t(z)z\,dz+\int_0^\infty f_t(z)z\,dz,
	\end{align*}
	where, as above, $f_t(z)=p_{t}(x-z)$.Thus
	\begin{align*}
		\partial_t r(t,x)
		&=
		-\int_\infty^0\partial_t f_t(z)z\,dz+\int_0^\infty\partial_t f_t(z)z\,dz
		=
		-\frac12\int_\infty^0f''_t(z)z\,dz+\frac12\int_0^\infty f''_t(z)z\,dz\\
		&=
		\frac12\int_\infty^0f'_t(z)\,dz-\frac12\int_0^\infty f'_t(z)\,dz=f_t(0)=p_t(x)
	\end{align*}
	Together with $r(0,x)=|x|$, this proves \eqref{eq:Brownian-identity}.
\end{proof}

\section{Proofs for the Ginzburg-Landau $\nabla$ interface model}

We shall prove Theorem~\ref{thm:aging-GL}. Let $u(t,k):=u_k(t), t\ge 0, k\in\bZ$. We need to establish the covariance-to-variance reduction for $u$,
which is by Lemma \ref{lem:stat-CVTV} a consequence of the next Lemma.
\begin{lemma}
For every $s\geq 0$ and $j\in\mathbb{Z}$, we have the identity
\begin{align*}
	u(\cdot,\cdot) - u(s,j)
    \overset{\mathrm{law}}{=}
    u(\cdot-s,\cdot-j).
\end{align*}
\end{lemma}
\begin{proof}
	We proceed as in the previous cases:
	\begin{align*}
		u(t,k)
		&=
		u(s,k)
		+
		 \frac12 \int^t_s \left(
		 	V'(\grad u_{k-1}(r))-V'(\grad u_k(r))
		 \right) \, dr
		 +
		 B_k(t)-B_k(s)
		 \\
		 &=
		 u(s,j)
		 \\
		 &
		 \,\,+
		 u(s,k)-u(s,j)
		+ \frac12
		 \int^t_s \left(
		 	V'(\grad u_{k-1}(r))-V'(\grad u_k(r))
		 \right) \, dr
		 +
		 B_k(t)-B_k(s)
	\end{align*}
	By shifting the space and time parameters, we can see that the last line is equal in law to
	\begin{align*}
		u(0,k-j)-u(0,0)
		+
		\frac12 \int^{t-s}_0 \left(
		 	V'(\grad u_{k-j-1}(r))-V'(\grad u_{k-j}(r))
		 \right) \, dr
		 +
		 B_{k-j}(t-s)-B_{k-j}(0).
	\end{align*}
	Note that we used the time and space invariance of the law of the process of the discrete gradients $\{\grad u_j(\cdot):\, j\in\ze\}$.
	The claim follows.
\end{proof}
We now sketch the proof of Theorem~\ref{thm:aging-GL}. For simplicity, we restrict to the case $V(x)=\frac{x^2}{2}$, the general case following along the same lines. We let $\ou_n(t,x) = n^{-1/4} u(c_1tn,c_2x\sqrt{n})$. Note that
\begin{align*}
	\CRvv \left( u(tn,y\sqrt{n}),\, u(sn,x\sqrt{n})\right)
	&=
	\CRvv \left( \ou_n(t,y),\, \ou_n(s,x)\right),
\end{align*}
so that, thanks to the covariance-to-variance reduction, the first statement of Theorem~\ref{thm:aging-GL} i.e.
\begin{align*}
	\lim_{n\to\infty} \cR^{GL}_V(s,as;x\sqrt{n},y\sqrt{s}) = \cR^{EW}(1,a;x,y)
\end{align*}
follows once we have the convergence
\begin{eqnarray}\label{eq:convergence-variances-GL}
	\lim_{n\to\infty} \Vvv \, \ou_n(t,x) = \Vvv \, \ou(t,x).
\end{eqnarray}
Proposition~\ref{thm:variance-GL} is a discrete analogue of Proposition~\ref{thm:variance-EW}. The convergence \eqref{eq:convergence-variances-GL} in this case is then a consequence of the invariance principle. We defer these last details to the end of the section.
\begin{proposition}\label{thm:variance-GL}
	Consider the case $V(x)=\frac{x^2}{2}$.
	We have the formula
	\begin{align*}
		\Vvv \, u(t,k)
		&=
		E_k[|X_t|],
	\end{align*}
	where, under $E_k$, $X$ is a simple symmetric continuous-time random walk starting at $k$.
\end{proposition}
\begin{proof}
	For the discrete additive SHE a corresponding equality holds, involving the discrete Laplacian. This can be verified in a straight forward manner using the covariance-to-variance reduction. Indeed, let
$$u(t,i)=u(0,i)+\frac12\int_0^t(u(s,i+1)+u(s,i-1)-2u(s,i))\,ds
+B_i(t),\qquad i\in\Bbb Z,$$
where $u(0,0)=0$ and $\{u(0,i+1)-u(0,i)\}_i$ are standard normal i.i.d.\ random variables.
By Ito's formula
$$du(t,i)^2=\big[u(t,i)(u(t,i+1)+u(t,i-1)-2u(s,i))+1\big]\,dt+2u(t,i)d B_i(t).$$
Let $f(t,i)=\Vvv\,(u(t,i))=E[u(t,i)^2]$, since $E[u(t,i)]=0$.
Then
\begin{align*}
\partial_tf(t,i)
&=
E[\big[u(t,i)(u(t,i+1)+u(t,i-1)-2u(t,i))+1\big]\\
&=
\Cvv\,(u(t,i),u(t,i+1))+\Cvv\,(u(t,i),u(t,i-1))
-2\Vvv\,(u(t,i))+1.
\end{align*}
In view of the covariance-to-variance reduction and $\Vvv\,(u(0,\pm1))=1$
\begin{align*}
\Cvv\,(u(t,i),u(t,i\pm1))
&=
\frac12 \Vvv\,(u(t,i))+\frac12 \Vvv\,(u(t,i\pm1))-\frac12 \Vvv\,(u(0,\mp1))\\
&=
\frac12 f(t,i)+\frac12 f(t,i\pm1)-\frac12.
\end{align*}
Thus
$$\partial_tf(t,i)=\frac12 \big[f(t,i+1)+f(t,i-1)-2f(t,i)\big]=\frac12 \Delta f(t,i),$$
where $\Delta$ is the discrete Laplacian and
$$f(0,i)=|i|.$$
Note that this is nothing but the stochastic heat equation for the continuous-time simple random walk $X(t)$:
$$f(t,i)=\sum_{j}p^{\text{RW}}_t(i-j)|j|=E_i[|X(t)|]=E_0[|X(t)+i|],$$
where $p^{\text{RW}}_t(j)=P_0(X(t)=j)$.
\end{proof}
We finish the proof of Theorem~\ref{thm:aging-GL}.
Let $f_n(t,x)=\Vvv\, \ou_n(t,x)$ with $\ou_n(t,x)=n^{-1/4}u(tn^2,\lfloor nx\rfloor)$,
$X_n(t)=n^{-1}X(tn^2)$ and $x_n=n^{-1}\lfloor nx\rfloor$.
By the invariance principle,
$$f_n(t,x)=E_{x_n}[|X_n(t)|]\to E_x[|B(t)|]=r(t,x),$$
as $n\to\infty$,
where $B(t)$ is a Brownian motion.

\section{Proofs for directed polymers}\label{sec:proof-polymers}


\subsection{The space-time stationary structure}\label{sec:proof-polymers-stationarity}

The covariance-to-variance reduction for directed polymers is a consequence of Proposition \ref{thm:space-time-stationarity-polymers} below.
At this point, it is convenient to work with $\beta=1$. The general case follows from Brownian scaling which yields the identity
\begin{align*}
	Z^{\beta,\para}(t,n)
\overset{\mathrm{law}}{=}
 \beta^{-2n} Z^{1,\beta^{-2}\theta}(\beta^2 t, n).
\end{align*}
Moreover, it is clear that the pre-factor will not affect the variances once we take the logarithm.
In the following, we abbreviate $Z^{\theta}=Z^{1,\theta}$.

We now describe the stationary structure of the model. Let $Z^{\para}(t,0):=e^{-B_0(t)+\theta t}$ and define processes $r_n(\cdot)$ and $Y_n(\cdot)$ for $n\geq 1$ as
\begin{align*}
 \log Z^{\theta}(t,n) &- \log Z^{\theta}(t,n-1) = r_n(t)\\
 \log Z^{\theta}(t,n) &- \log Z^{\theta}(s,n) = \theta(t-s)-Y_n(s,t).
\end{align*}
In particular, we have the identity
\begin{eqnarray*}
	\log Z^{\para}(t,n) = -B_0(t)+\para t + \sum^n_{k=1}r_k(t).
\end{eqnarray*}
The following Lemma summarizes the Burke's property from \cite{OY} and parts of \cite[Theorem 3.3]{SV}:
\begin{lemma}\label{thm:stationary-structure}
	The family of processes $\{r_n, Y_n:\, n\geq 1\}$ satisfies the following properties:
	\begin{enumerate}
		\item[a.-] For each fixed $t\ge 0$,  the random variables $(r_k(t))_{k\ge 1}$ are
i.i.d.\ and, for each $k\geq 1$, $e^{-r_k(t)}$ follows a Gamma$(\theta)$ distribution.

		\vspace{1ex}

		\item[b.-]  For each fixed $n$, $\{Y_n(0,t):\, t\in\re\}$ is a two-sided Brownian motion.
		
		\vspace{1ex}
		
		\item[c.-]  For each $n\geq 1$ and each $-\infty < s_1 \leq s_2 \leq \cdots \leq s_n < \infty$, the process and the random variables
				\begin{align*}
					\{Y_n(0,t):\, t\leq s_1 \}, \quad r_k(s_k),\, k=1,\cdots, n,
				\end{align*}
				are independent.
	\end{enumerate}
\end{lemma}
By the discussion at the opening of Section~\ref{sec:results-open-questions}, the discrete covariance-to-variance reduction in Lemma~\ref{thm:CVTV-polymers} follows from the following Proposition.
\begin{proposition}\label{thm:space-time-stationarity-polymers}
	Let $s\geq 0$ and $m\geq 0$. Then, we have the identity
	\begin{align*}
		\frac{Z^{\theta}(s+\cdot,m+\cdot)}{Z^{\theta}(s,m)}
		\overset{\mathrm{law}}{=}
		Z^{\theta}(\cdot,\cdot).
	\end{align*}
\end{proposition}
\begin{proof}
	The proof is similar to the case of SHE. We first note that
	\begin{align*}
		Z^{\theta}(s+t,m+n)
		&=
		\int_{-\infty < s_m < \cdots < s_{m+n-1} < s+t}
		Z^{\theta}(s_m,m)e^{\sum^n_{j=m+1} B_j(s_{j-1},s_j)} ds_m \cdots ds_{m+n-1},
	\end{align*}
	with the convention $s_{m+n}=s+t$. We let
	\begin{align*}
		e^{W(s,s+u)}
		:=
		\frac{Z^{\theta}(s+u,m)}{Z^{\theta}(s,m)}e^{-\theta u},
	\end{align*}
	and notice that, by Lemma~\ref{thm:stationary-structure}, $\tilde W(\cdot):=W(s,s+\cdot)$ is a two-sided Brownian motion which is independent of $\tilde{B}_j(\cdot) := B_{j+m}(s+\cdot)$ for all $j\ge 1$.
	Hence, by the change of variables $s_j\to s+s_j$ followed by setting $\tilde s_j = s_{j+m}$, we get the following 
	\begin{align*}
		\frac{Z^{\theta}(s+t,m+n)}{Z^{\theta}(s,m)}
		&=
		\int_{-\infty < s_m < \cdots < s_{m+n-1} < s+t}
		e^{W(s,s_m)+\theta(s_m-s)}e^{\sum^n_{j=m+1} B_j(s_{j-1},s_j)} ds_m \cdots ds_{m+n-1}
		\\
		&=
		\int_{-\infty < s_m < \cdots < s_{m+n-1} < t}
		e^{W(s,s + s_m)+\theta s_m}e^{\sum^n_{j=m+1} B_j(s+s_{j-1},s+s_j)} ds_m \cdots ds_{m+n-1}\\
		&=
		\int_{-\infty < \tilde s_0 < \cdots < \tilde s_{n-1} < t}
		e^{\tilde W(\tilde s_0)+\theta \tilde s_0} e^{\sum^{n-m}_{j=1} \tilde B_j(\tilde s_{j-1},\tilde s_j)} d \tilde s_0 \cdots d \tilde s_{n-1},
	\end{align*}
	where $s_n=\tilde{s}_{n-m}=t$, and the last equality should be understood as an equality of \emph{processes} on $\{(t,n):t\in\bR_+ n \in \bZ_+\}:$. To conclude, note that the last term has the same law of the right hand side of the required identity.
\end{proof}


\subsection{Uniform integrability and the proof of Theorem~\ref{thm:main-polymers}  }\label{sec:proof-polymers-UI}


Let $\kpz_n = \log \she_n^{\textsc{st}}$.  The proof of Theorem \ref{thm:convergence-moments-polymers} (and hence of Theorem~\ref{thm:main-polymers}) boils down to show that, for each $t\geq 0$, $x\in\re$ and each $p>0$, the family $\{\kpz_n(t,x)^p:\, n\geq 1\}$ is uniformly integrable. We can then take the limit on both sides of the discrete covariance-to-variance reduction Lemma~\ref{thm:CVTV-polymers}. This proves Theorem~\ref{thm:main-polymers}.

The proof of the uniform integrability of $\{\kpz_n(t,x)^p:\, n\geq 1\}$ will be based on the following elementary bound: for each $p>0$, there exists $C=C(p)$ such that
\begin{align*}
	|\kpz_n(t,x)|^p
	\leq
	C
	\left(
		\she^{\textsc{st}}_n(t,x) + \she^{\textsc{st}}_n(t,x)^{-1}
	\right).
\end{align*}
It is then enough to show that the expected value of $\she_n^{\textsc{st}}(t,x)^{-1}$ is uniformly bounded in $n$. This is the content of the following theorem.
\begin{theorem}\label{thm:uniform-lower-tail}
	For each fixed $t\geq 0$ and $x\in \re$, we have
	\begin{eqnarray*}
		\sup_{n\geq 1}
		\esp\left[
			\she_n^{\textsc{st}}(t,x)^{-1}
		\right]
		<\infty.
	\end{eqnarray*}
\end{theorem}

We give the proof of the theorem after stating the corresponding result for the point-to-point rescaled partition function which will then be proved in Section~\ref{sec:proof-polymers-lower-tail-bounds} by means of Gaussian deviation bounds.
Note that we have the relation
\begin{align*}
	\she_n^{\textsc{st}}(t,x)
	=
	\int^{\sqrt{n} t}_{-\infty} e^{-\beta_n B_0(\sqrt{n}s)} \she_n(0,y;t,x)\, dy,
\end{align*}
with
{\small\begin{align}\label{eq:rescaled-ptp}
	\she_n(s,y;t,x)
	&
	=
	\sqrt{n}e^{-n(t-s)+\sqrt{n}(x-y)}e^{-\frac{1}{2}[\sqrt{n}(t-s)-(x-y)]}
	\,
	Z^{\beta_n}(ns-\sqrt{n}y,ns+1;nt-\sqrt{n}x,nt),
\end{align}}
where the point-to-point partition function in fixed temperature was defined in \eqref{eq:pf-ptp}.
\begin{theorem}\label{thm:uniform-lower-tail-ptp}
	For each fixed $t\geq 0$ and $a\geq 0$, there exists $c=c(t,a) \in (0,\infty)$, $C=C(t,a) \in (0,\infty)$ and $u_0=u_0(t,a)\geq 0$ such that
	\begin{align*}
		\p\left[
			\she_n(0,y;t,x) \leq C e^{-cu}
		\right]
		\leq
		2e^{-\frac{1}{2}u^2},
	\end{align*}
	for all $n\geq 1$, $|y-x|\in[-a,a]$ and $u\geq u_0$.
	
	As a consequence, for each $t\geq 0$, $p>0$ and $a\geq 0$, there exists $K=K(t,p,a)<\infty$ such that
	\begin{align*}
		\esp\left[
			\she_n(0,y;t,x)^{-p}
		\right]
		\leq K,
	\end{align*}
	for all $|y-x|\in[-a,a]$.
\end{theorem}
Note that, by translation invariance, it is enough to show the bound for $\she_n(t,x):=\she_n(0,0;t,x)$, uniformly in $x\in[-a,a]$.

This kind of bounds dates back at least to \cite[Theorem 2.1]{T-SK} in the context of the Sherrington-Kirkpatrick model (see also \cite[Theorem 2.2.7]{T-spin}). They were shown for the Hopfield model in \cite[Theorem 1.1]{T-hopfield}. In the context of directed polymers, they were obtain for Gaussian (resp. bounded) environments in \cite[Theorem 1.5]{CH} (resp. \cite[Proposition 1]{M-esbp}), and for a Brownian polymer in a Gaussian environment in \cite[Proposition 3.3]{RT}. All these results are for fixed temperature and fixed end-point. The result for directed discrete polymers in Gaussian environments in the intermediate disorder regime and locally uniformly in the end-point was obtained in \cite[Theorem 1]{M-positivity}.

Our proof is a blend of the approaches in \cite{RT} and \cite{M-positivity}.

\begin{proof}[Proof of Theorem~\ref{thm:uniform-lower-tail} assuming Theorem~\ref{thm:uniform-lower-tail-ptp}]
	We take $x=0$ to simplify the notation. The proof for a general $x\in \re$ is identical.
Recall the relation
\begin{align*}
	\she_n^{\textsc{st}}(t,0)
	=
	\int^{\sqrt{n} t}_{-\infty} W_n(y) \she_n(0,y;t,0)\, dy
	\quad
	\text{where}
	\quad
	W_n(y)
	=
	e^{-\beta_n B_0(\sqrt{n}y)}.
\end{align*}

Then, for all $a>0$, we have the bound
\begin{align*}
	\she_n^{\textsc{st}}(t,0)
	\geq
	\int^{a}_{-a} W_n(y) \she_n(0,y;t,0)\, dy.
\end{align*}
We let $W_n = W_{n,a} = \int^{a}_{-a} W_n(y)\, dy$. Then,
\begin{align*}
	\esp\left[
		\she_n^{\textsc{st}}(t,0)^{-1}
	\right]
	&\leq
	\esp\left[
		\left(
			\int^a_{-a} W_n(y) \she_n(0,y;t,0)\, dy
		\right)^{-1}
	\right]
	\\
	&=
	\esp\left[ W_n^{-1}
		\left(
			W_n^{-1}
			\int^a_{-a} W_n(y) \she_n(0,y;t,0)\, dy
		\right)^{-1}
	\right]
	\\
	&\leq
	\esp\left[ W_n^{-2}
			\int^a_{-a} W_n(y) \she_n(0,y;t,0)^{-1}\, dy
	\right]
	\\
	&\leq
	\esp\left[ W_n^{-4}\right]^{1/2}
	\esp\left[
				\left(
					\int^a_{-a} W_n(y) \she_n(0,y;t,0)^{-1}\, dy
				\right)^2
	\right]^{1/2}
\end{align*}
where we used Jensen's inequality with respect to the measure with density $W_n^{-1}W_n(y)$ to go from the second to the third line. The first expected value in the last line is uniformly bounded. Next, using Cauchy-Schwarz inequality twice,
\begin{align*}
	&
	\esp\left[
				\left(
					\int^a_{-a} W_n(y) \she_n(0,y;t,0)^{-1}\, dy
				\right)^2
	\right]
	\\
	&\leq
	\esp\left[
			\int^a_{-a} W_n(y)^2 dy
			\int^a_{-a} \she_n(0,y;t,0)^{-2}\, dy
	\right]
	\\
	&\leq
	\esp\left[
			\left(
				\int^a_{-a} W_n(y)^2 dy
			\right)^2
	\right]^{1/2}
	\esp\left[
			\left(
				\int^a_{-a} \she_n(0,y;t,0)^{-2}\, dy
			\right)^2
	\right]^{1/2}.
\end{align*}
The first expected value above is uniformly bounded. Finally, using Jensen's inequality once again,
\begin{align*}
	\esp\left[
			\left(
				\int^a_{-a} \she_n(0,y;t,0)^{-2}\, dy
			\right)^2
	\right]
	& \leq
	a \,
	\esp\left[
				\int^a_{-a} \she_n(0,y;t,0)^{-4}\, dy
	\right]
	\\
	&
	= \,
	a
	\int^a_{-a}
	\esp\left[
		\she_n(0,y;t,0)^{-4}
	\right]
	\, dy,
\end{align*}
which is uniformly bounded in virtue of Theorem~\ref{thm:uniform-lower-tail-ptp}.
\end{proof}


\subsection{Gaussian deviation bounds}\label{sec:proof-polymers-deviation-bounds}


In this section, we show a Gaussian deviation bound that will be the key to the proof of Theorem~\ref{thm:uniform-lower-tail-ptp}. To this end, we will rely on Gaussian concentration estimates based on Malliavin calculus.
From now on, we specify our probability space. We let $\Omega$ be the space of continuous real valued functions defined on $\re_+ \times \ze_+$ with the cylindrical $\sigma$-algebra and we let $\p$ be the standard Wiener measure on $\Omega$. For each fixed $B\in \Omega$, we then define the environment $\{B_k:k\geq 1\}$ as
\begin{eqnarray*}
	B_k(t)= B(t,k).
\end{eqnarray*}
We also consider the space
\begin{eqnarray*}
	H_1
	=
	\left\{
		h:\re_+ \times \ze_+ \to \re:
		\,
		\| h \|_{H_1}^2
		:=
		\sum_k \int_{\re_+} |\dot{h}_k(s)|^2 ds		
		< \infty
	\right\}.
\end{eqnarray*}
The triple $(\Omega,H_1,\p)$ is known as the standard Wiener space.
For a measurable set $A\subset \Omega$ and $B\in \Omega$, we define
\begin{eqnarray*}
	q_A(B) = \inf \left\{ \| h\|_{H_1} :\, B+h \in A \right\}.
\end{eqnarray*}
The main estimate of this section is:
\begin{lemma}\label{thm:deviation-bound-distance}
	For each $p>0$, there exists a constant $c_p \in (0, \infty)$ such that
	\begin{eqnarray*}
		\p\left[
			q_A > c_p + u
		\right]
		\leq
		2 e^{-\frac{u^2}{2}},
	\end{eqnarray*}
	for all $u>0$ and all measurable set $A\subset \Omega$ such that $\p[A] \geq p$.
\end{lemma}
Before turning to the proof, we need to introduce some tools from Malliavin calculus. We say that a function $F:\Omega \to \re$ is cylindrical if there exists $n\geq 1$, $f\in C^1_0(\re^n;\re)$, $k_1,\cdots,k_n \geq 1$ and $t_1,\cdots,t_n \geq 0$ such that
\begin{eqnarray*}
	F(B) = f(B_{k_1}(t_1),\cdots, B_{k_n}(t_n)).
\end{eqnarray*}
For a cylindrical function $F$ and $h\in H_1$, we define
\begin{eqnarray*}
	D_h F(B) = \frac{d}{d\epsilon} F(B + \epsilon h) |_{\epsilon=0}.
\end{eqnarray*}
It easily follows that
\begin{eqnarray*}
	D_h F(B)
	=
	\sum^n_{j=1}
	\partial_j f(B_{k_1}(t_1),\cdots, B_{k_n}(t_n)) h(t_j,k_j).
\end{eqnarray*}
Hence, for each $B\in\Omega$ and each cylindrical function $F$, the mapping $h \mapsto D_h F(B)$ defines a continuous linear functional on $H_1$. As a consequence, for each $B\in\Omega$ and each cylindrical function $F$, there exists a unique $D F(B) \in H_1$ such that
\begin{eqnarray*}
	\langle D F(B),h \rangle_{H_1} = D_h F(B).
\end{eqnarray*}
From \cite[Proposition I.1]{Ust},
the operator $D$ can be extended to a continuous linear functional from $L^p(\Omega,\p;\re)$ to $L^p(\Omega,\p;H_1)$ for all $p>1$. We then define $\mathbb{D}_{p,1}$ as the space of functions $F\in L^p(\Omega,\p;\re)$ such that
\begin{eqnarray*}
	\| F \|_{p,1}
	:=
	\| F \|_{L^p(\Omega,\p;\re)} + \| DF \|_{L^p(\Omega,\p;H_1)}
	<
	\infty.
\end{eqnarray*}
We are now ready to state the key Gaussian concentration inequality which corresponds to \cite[Theorem 1, p.70]{Ust}.
\begin{theorem}\label{thm:gaussian-concentration}
	Let $F\in \mathbb{D}_{p,1}$ for some $p>1$ and suppose that $DF \in L^{\infty}(\Omega,\p;H_1)$. Let $m = \esp[F]$ and $\sigma^2 = \| DF\|^2_{L^{\infty}(\Omega,\p;H_1)}$. Then,
	\begin{eqnarray*}
		\p\left[
			|F-m| > u
		\right]
		\leq
		2 e^{-\frac{u^2}{2\sigma^2}},
	\end{eqnarray*}
	for all $u>0$.
\end{theorem}
We need one more ingredient:
\begin{lemma}\label{thm:finitness-q}
	If $\p[A]>0$, then $q_A$ is $\p$-almost surely finite.
\end{lemma}
\begin{proof}
	Let $J=\{q_A<\infty\}$.
	We use the following elementary fact \cite[Proposition 1.2.6]{Nu}: ${\bf 1}_J \in \mathbb{D}^{1,1}$ if and only if $\p[J]=0$ or $1$.
	
	Observe that $J=J+H_1$. Hence, for any $h\in H_1$ and $B\in\Omega$, ${\bf 1}_{J}(B+h)={\bf 1}_J(B)$, so that $D{\bf 1}_J=0$.
	On the other hand, $\| {\bf 1}_J \|_{L^1(\Omega,\p;\re)}=\p[J]<\infty$. Hence, ${\bf 1}_J \in \mathbb{D}^{1,1}$ and $\p[J]=0$ or $1$. As $\p[J] \geq \p[A]>0$, we necessarily have that $\p[J]=1$.
\end{proof}
We can now complete the proof of Lemma~\ref{thm:deviation-bound-distance}.
\begin{proof}[Proof of Lemma~\ref{thm:deviation-bound-distance}]
	Let $h\in H_1$ and $\epsilon>0$. By the triangle inequality and the previous lemma, we have that $|q_A(B+\epsilon h)- q_A(B)| \leq \epsilon \| h \|_{H_1}$ for all $B\in\Omega$. Hence, $|D_h q_A(B)| \leq \| h \|_{H_1}$ for all $B\in\Omega$ and $h\in H_1$ so that $|\langle D q_A(B), h \rangle_{H_1}| \leq \| h \|_{H_1}$ for all $B \in \Omega$ and all $h\in H_1$. As a consequence,
	\begin{eqnarray*}
		 \| D q_A(B) \|_{H_1} \leq 1,
	\end{eqnarray*}
	for all $B\in\Omega$.
	
	For each $M\geq 1$, we introduce a cut-off function $f_M:\re_+ \to \re_+$ such that $f_M(x)=x$ for $x\in[0,M]$, $f_M(x)=0$ for $x\geq 2M+1$, $f_M(x) \leq x$ for all $x$, $\| f'_M \|_{\infty}\leq 1$ and such that $f_M \leq f_{M+1}$. We then define $q^M_A = f_M \circ q_A$. By the chain rule,
	\begin{eqnarray*}
		\| D q^M_A(B) \|_{H_1}
		=
		\| f'_M(q_A(B)) D q_A(B) \|_{H_1}
		\leq 1.
	\end{eqnarray*}
	As $q^M_A$ is bounded, we have $q^M_A \in \mathbb{D}^{1,p}$ for any $p>1$.
	By Theorem~\ref{thm:gaussian-concentration}, we then have
	\begin{eqnarray*}
		\p\left[
			|q^M_A-\esp[q^M_A] > u|
		\right]
		\leq
		2 e^{-\frac{u^2}{2}},
	\end{eqnarray*}
	for all $u>0$. Assume now that $\p[A]\geq p >0$. Then, for all $u<\esp[q^M_A]$,
	\begin{eqnarray*}
		p
		\leq
		\p[A]
		\leq
		\p\left[
			|q^M_A-\esp[q^M_A(\omega)]|>u
		\right]
		\leq
		2 e^{-\frac{u^2}{2}}.
	\end{eqnarray*}
	It follows that $\esp[q^M_A] \leq c_p:=(2 \log(2/p))^{1/2}$. Using Theorem~\ref{thm:gaussian-concentration} once again, we conclude that
	\begin{eqnarray*}
		\p\left[
			q^M_A > c_p + u
		\right]
		\leq2 e^{-\frac{u^2}{2}},
	\end{eqnarray*}
	for all $u>0$ and all $M\geq 1$. The result follows by Fatou's lemma.
\end{proof}


\subsection{Some preliminaries}\label{sec:proof-polymers-preliminaries}

Recall the definition of the point-to-point partition function \eqref{eq:pf-ptp}.
We shall use the notation $B(s,i):=B_{i}(s), i\ge 0$. Let $X_{\cdot}$ denote a rate one Poisson process with $X_0=1$, let $P$ denote its law and let $E$ be the expected value with respect to $P$. Then, the partition function can be written as
\begin{eqnarray*}
	Z^{\beta}(s,m;t,n) = e^{t-s} E[e^{\beta H_{s,t}(X)} {\bf 1}_{X_t=n}|X_s=m+1],
\end{eqnarray*}
where
\begin{eqnarray*}
	H_{s,t}(X) = \int^t_s dB_{X_u}(u) = \int^t_s dB(u,X_u).
\end{eqnarray*}
For readability in this section we shall use the mild abuse of notation
\begin{eqnarray*}
	Z^{\beta}(t,n) := Z^{\beta}(0,1;t,n) , \,\, H_t(X):= H_{0,t}(X),
\end{eqnarray*}
(we stress that this differs from the one used in Section~\ref{sec:proof-polymers-stationarity}).
For two paths $X$ and $\tilde{X}$, we define their overlap as
\begin{eqnarray*}
	L_t(X,\tilde{X}) = \int^t_0 {\bf 1}_{X_s = \tilde{X}_s}ds.
\end{eqnarray*}
We list some elementary identities:
\begin{lemma}\label{thm:simple-identities}
	For all $t>0$ and $n\geq 1$, we have
	\begin{align*}
		\esp[Z^{\beta}(t,n)]
		&=
		e^{t}e^{\frac{\beta^2}{2}t}P[X_t=n],
		\\
		\esp[Z^{\beta}(t,n)^2]
		&=
		e^{2t + \beta^2 t}
		E^{\otimes 2}[e^{\beta^2 L_t(X,\tilde{X})}{\bf 1}_{X_t = \tilde{X}_t=n}],
		\\
		\frac{Z^{\beta}(t,n)^2}{\esp[Z^{\beta}(t,n)]^2}
		&=
		E^{\otimes 2}[e^{\beta^2 L_t(X,\tilde{X})} | {X_t = \tilde{X}_t=n}].
	\end{align*}
\end{lemma}
In the following, we write $Z^{\beta}_B(t,n)$ and $H^{B}_t(X)$ to stress the dependence on the environment $B=\{B(s,i),s\in\bR, i\ge 1\}$. We also abbreviate the polymer measure in the environment $B$ by $\langle \cdot \rangle_{t,n,\beta,B}$.
\begin{lemma}
	Assume $B = \bar{B} + h$ with $h\in H_1$. Then,
	\begin{eqnarray*}
		\log Z^{\beta}_B(t,n)
		\geq
		\log Z^{\beta}_{\bar{B}}(t,n)
		- \beta \sqrt{ \langle L_t (X,\tilde{X}) \rangle^{\otimes 2}_{t,n, \beta,\bar{B}}}\| h \|_{H_1}.
	\end{eqnarray*}
\end{lemma}
\begin{proof}
	Using that $B = \bar{B} + h$
	\begin{align*}
		e^{-t}Z^{\beta}_B(t,n)
		&=
		E[e^{\beta H^{B}_t(X)}{\bf 1}_{X_t=n}]
		=
		E[e^{\beta H^{\bar{B}}_t(X)}e^{\beta H^h_t(X)}{\bf 1}_{X_t=n}]
		\\
		&=
		e^{-t} Z^{\beta}_{\bar{B}}(t,n)
		\left\langle
			e^{\beta H^h_t(X)}
		\right\rangle_{t,n,\beta,B}
		\\
		&\geq
		e^{-t} Z^{\beta}_{\bar{B}}(t,n)
		e^{
		\beta
		\left\langle
			H^h_t(X)
		\right\rangle_{t,n,\beta,B}
		}
	\end{align*}
	Now,
	\begin{align*}
		&
		\left|
			\left\langle
				H^h_t(X)
			\right\rangle_{t,n,\beta,B}
		\right|
		=
		\left|
			\left\langle
				\int^t_0 \dot{h}(s,X_s)ds
			\right\rangle_{t,n,\beta,B}
		\right|
		\\
		&=
		\left|
			\left\langle
				 \sum_k\int^t_0 {\bf 1}_{X_s=k}\dot{h}(s,k)ds
			\right\rangle_{t,n,\beta,B}
		\right|
		=
		\left|
			 \sum_k\int^t_0
			\left\langle
				 {\bf 1}_{X_s=k}
			\right\rangle_{t,n,\beta,B}
			\dot{h}(s,k)ds
		\right|
		\\
		&\leq
		\left|
			 \sum_k\int^t_0
			\left\langle
				 {\bf 1}_{X_s=k}
			\right\rangle_{t,n,\beta,B}^2
		\right|^{1/2}
		\times
		\| h \|_{H_1}
		=
		\sqrt{
		\left\langle	
			L_t(X,\tilde{X})
		\right\rangle_{t,n,\beta,B}^{\otimes 2}
		}
		\times
		\| h \|_{H_1}.
	\end{align*}
\end{proof}


\subsection{Proof of the lower tail bounds}\label{sec:proof-polymers-lower-tail-bounds}


In the following, we write $L_{t,x,n}(X,\tilde{X})=L_{tn-x\sqrt{n}}(X,\tilde{X})$ and
$\langle \cdot \rangle_{t,x,n,B}=\langle \cdot \rangle_{tn-x\sqrt{n},tn,\beta_n,B}$.
For $K>0$, we define the event
\begin{eqnarray*}
	A_{n}(t,x,K)
	=
	\left\{
		B: \, \she_{n,B}(t,x) \geq \mfrac12 \esp[\she_n(t,x)],\,
		\langle L_{t,x,n}(X,\tilde{X}) \rangle_{t,x,n,B} \leq K \sqrt{n}
	\right\},
\end{eqnarray*}
where we denoted $\she_n(t,x):=\she_n(0,0;t,x)$ the point-to-point partition function defined in \eqref{eq:rescaled-ptp}.
%
\begin{lemma}\label{thm:uniform-lower-bound-probability}
	For all $a>0$, $t>1$ and $K$ large enough, there exists $\delta = \delta(a,t,K)>0$ such that
	\begin{eqnarray*}
		\p[A_{n}(t,x,K)] \geq \delta,
	\end{eqnarray*}
	for all $n\geq 1$ and $|x|\leq a$.
\end{lemma}
\begin{proof}
	We denote $E_{t,x,n}[\cdot] = E[\cdot | X_{tn }=tn+x\sqrt{n}]$ and $H_{t,x,n}(X)=H_{tn-x\sqrt{n}}(X)$.
	Then,
	\begin{align*}
		\p[A_{n}(t,x,K)]
		&=
		\p\Big[
			\she_n(t,x) \geq \mfrac12 \esp[\she_n(t,x)],
			\\
			&
			\quad
			 E_{t,x,n}^{\otimes 2}\left[
				L_{t,x,n}(X,\tilde{X})e^{\beta_n({H}_{t,x,n}(X)+{H}_{t,x,n}(\tilde{X}))}
			\right]
			\leq K \sqrt{n} {\she}_n(t,x)^2
		\Big]
		\\
		&\geq
		\p\Big[
			\she_n(t,x) \geq \mfrac12 \esp[\she_n(t,x)],
			\\
			&
			\quad
			E_{t,x,n}^{\otimes 2}\left[
				L_{t,x,n}(X,\tilde{X})e^{\beta_n({H}_{t,x,n}(X)+{H}_{t,x,n}(\tilde{X}))}
			\right]
			\leq
			\frac{K \sqrt{n}}{4} \esp[{\she}_n(t,x)]^2
		\Big]
		\\
		&\geq
		\p\left[
			\she_n(t,x) \geq \mfrac12 \esp[\she_n(t,x)]
		\right]
		-1
		\\
		&
		+
		\p\left[
			E_{t,x,n}^{\otimes 2}\left[
				L_{t,x,n}(X,\tilde{X})e^{\beta_n({H}_{t,x,n}(X)+{H}_{t,x,n}(\tilde{X}))}
			\right]
			\leq
			\frac{K \sqrt{n}}{4} \esp[{\she}_n(t,x)]^2
		\right]
	\end{align*}
	Now, by Lemma~\ref{thm:simple-identities} and~\ref{thm:estimates-Poisson},
	\begin{eqnarray*}
		\frac{\esp[\she_n(t,x)^2]}{\esp[\she_n(t,x)]^2}
		=
		E^{\otimes 2}_{t,x,n}[e^{2\beta_n^2 L_{t,x,n}(X,\tilde{X})}]
		\leq C_1,
	\end{eqnarray*}
	for some finite $C_1=C_1(a)$ and for all $|x|\leq a$, $n\geq 1$.
	Hence, by Paley-Zygmund's inequality,
	\begin{eqnarray*}
		\p\left[
			\she_n(t,x) \geq \mfrac12 \esp[\she_n(t,x)]
		\right]
		\geq
		\frac{\esp[\she_n(t,x)]^2}{\esp[\she_n(t,x)^2]}
		\geq
		\frac{1}{4C_1},
	\end{eqnarray*}
	for all $|x|\leq a$ and $n\geq 1$.
	Now, by Chebyshev's inequality,
	\begin{align*}
		&
    \p\left[
			E_{t,x,n}^{\otimes 2}\left[
				L_{t,x,n}(X,\tilde{X}) e^{\beta_n({H}_{t,x,n}(X)+{H}_{t,x,n}(\tilde{X}))}
			\right]
			>
			\frac{K \sqrt{n}}{4} \esp[{\she}_n(t,x)]^2
		\right]
		\\
		& \leq
		\frac{4}{K\sqrt{n} \esp[{\she}_n(t,x)]^2}
		\esp\left[
			E_{t,x,n}^{\otimes 2}\left[
				L_{t,x,n}(X,\tilde{X})e^{\beta_n({H}_{tn}(X)+{H}_{tn}(\tilde{X}))}
			\right]
		\right]
		\\
		&=
		\frac{4}{K } \frac{1}{\sqrt{n}}
			E_{t,x,n}^{\otimes 2}\left[
				L_{t,x,n}(X,\tilde{X})e^{\beta_n^2 L_{t,x,n}(X,\tilde{X})}
			\right],
	\end{align*}
	which is finite by Lemma~\ref{thm:estimates-Poisson}, uniformly in $|x|\leq a$ and $n\geq 1$. The proof follows by taking $K$ large enough.
\end{proof}
We can now complete the proof:
\begin{proof}[Proof of Theorem~\ref{thm:uniform-lower-tail-ptp}]
	By translation invariance, it is enough to obtain uniform deviation bounds on $\she_n(t,x)$.
	Recall that, if $B=\bar{B}+h$ with $h\in H_1$, then
	\begin{eqnarray*}
		\log \she_{n,B}(t,x)
		\geq
		\log \she_{n,\bar{B}}(t,x)
		- \beta_n \sqrt{ \langle L_{t,x,n} (X,\tilde{X}) \rangle^{\otimes 2}_{t,x,n, \bar{B}}} \| h \|_{H_1}.
	\end{eqnarray*}
	If $\bar{B}\in A_n(t,x,K)$, this further yields
	\begin{eqnarray*}
		\log \she_{n,B}(t,x)
		\geq
		\log \esp[\she_n(t,x)]-\log 2
		- \sqrt{K} \| h \|_{H_1}.
	\end{eqnarray*}
	Hence,
	\begin{eqnarray*}
		\log \she_{n,B}(t,x)
		\geq
		\log \esp[\she_n(t,x)]-\log 2
		- \sqrt{K}q_{A_n(t,x,K)}(B).
	\end{eqnarray*}
	Then, for appropriate constants $c_1$ and $c_2$, we have
	\begin{eqnarray*}
		\p\Big[
			\log \she_n(t,x) < \log \esp[\she_n(t,x)]-c_1-c_2 u
		\Big]
		\leq
		\p\left[
			q_{A_n(t,x,K)} > c_p + u
		\right]
		\leq
		2 e^{-\frac{u^2}{2}},
	\end{eqnarray*}
	where the use of Proposition~\ref{thm:deviation-bound-distance} is justified by Lemma~\ref{thm:uniform-lower-bound-probability}.
\end{proof}



\section*{Acknowledgments}
The authors owe their gratitude to Ivan Corwin for providing us with very
valuable suggestions and sharing us with his recent progress on this
topic and to Kostya Khanin for a crucial observation on stationary KPZ
increments. They are thankful to Nicolas Perkowski for supplying to the latter a rigourous proof and for numerous useful comments. TO is grateful to Tommaso Cornelis Rosati for valuable discussions. JDD is thankful to Tadahisa Funaki for valuable remarks.
The work of GMF was partially supported by Fondecyt grant 1171257, N\'ucleo Milenio `Modelos Estoc\'asticos de Sistemas Complejos y Desordenados' and MATH Amsud `Random Structures and Processes in Statistical Mechanics'.
The work of JDD and TO was supported by the German Research Foundation via DFG Research Unit FOR2402.


\appendix

\section{Estimates on the Poisson process}

Recall that we denote $E_{t,x,n}[\cdot] = E[\cdot | X_{tn-x\sqrt{n}}=tn]$ and $L_{t,x,n}(X,\tilde{X})=L_{tn-x\sqrt{n}}(X,\tilde{X})$.
\begin{lemma}\label{thm:estimates-Poisson}
	For all $a>0$ and $t>0$, we have
	\begin{align*}
		\sup_{|x|\leq a}
		E^{\otimes 2}_{t,x,n}\left[
			e^{\beta_n^2 L_{t,x,n}(X,\tilde{X})}
		\right]
		&< \infty,
		\\
		\sup_{|x|\leq a} \beta_n^2
		E^{\otimes 2}_{t,x,n}\left[ L_{t,x,n}(X,\tilde{X})
			e^{\beta_n^2 L_{t,x,n}(X,\tilde{X})}
		\right]
		&< \infty.
	\end{align*}
\end{lemma}
\begin{proof}
    First, it is enough to prove the first statement for all fixed $t>0$ and $a>0$, as then the second statement follows by Cauchy-Schwartz inequality.
	Write $s=tn-x\sqrt{n}$ and $m=tn$ for simplicity and assume without loss of generality that $m$ is an integer.
	We shall reduce the problem to estimate the overlap over half of the trajectories.
	Let
	\begin{align*}
		L'_{s/2}(X,\tilde{X})=\int^{s}_{s/2}{\bf 1}_{X_r = \tilde{X}_r}dr,
	\end{align*}
	and note that under $E_{t,x,n}$, $L_{s/2}(X,\tilde{X})$ and $L'_{s/2}(X,\tilde{X})$ have the same law, by considering the processes backwards in time and recalling the jump times have Lebesgue measure zero. Cauchy-Schwartz inequality then implies
	\begin{align*}
		E^{\otimes 2}_{t,x,n}\left[
			e^{\beta_n^2 L_{t,x,n}(X,\tilde{X})}
		\right]
		&=
		E^{\otimes 2}_{t,x,n}\left[
			e^{\beta_n^2( L_{s/2}(X,\tilde{X})+L'_{s/2}(X,\tilde{X}))}
		\right]
		\leq
		E^{\otimes 2}_{t,x,n}\left[
			e^{2\beta_n^2 L_{s/2}(X,\tilde{X})}
		\right].
	\end{align*}
	Next, we will de-condition the trajectories: let $p(t;k,l)=P[X_t=l| X_0=k]$, then
	\begin{align*}
		&
		E^{\otimes 2}_{t,x,n}\left[
			e^{2\beta_n^2 L_{s/2}(X,\tilde{X})}
		\right]
		=
		\frac{
		\sum^{m}_{k,	\tilde{k}=1}
		E^{\otimes 2}\left[
			e^{2\beta_n^2 L_{s/2}(X,\tilde{X})}
			{\bf 1}_{X_{s/2}=k}
			{\bf 1}_{\tilde{X}_{s/2}=\tilde{k}}
			{\bf 1}_{X_s = \tilde{X}_s=m}
		\right]
		}
		{
		p(s;0,m)^2
		}
		\\
		&
		=
		\frac{
		\sum^{m}_{k,	\tilde{k}=1}
		E^{\otimes 2}\left[
			e^{2\beta_n^2 L_{s/2}(X,\tilde{X})}
			{\bf 1}_{X_{s/2}=k}
			{\bf 1}_{\tilde{X}_{s/2}=\tilde{k}}
		\right]
		p(s/2;k,m)p(s/2;\tilde{k},m)
		}
		{
		p(s;0,m)^2
		}
		\\
		&
		\leq
		\max_{k=1,\dots,m}\frac{p(s/2;k,m)^2}{p(s;0,m)^2}
		E^{\otimes 2}\left[
			e^{2\beta_n^2 L_{s/2}(X,\tilde{X})}
			{\bf 1}_{X_{s/2},\tilde{X}_{s/2} \leq m}
	\right]
    \leq C(a) E^{\otimes 2}\left[
			e^{2\beta_n^2 L_{s/2}(X,\tilde{X})}
			{\bf 1}_{X_{s/2},\tilde{X}_{s/2} \leq m}
			\right],
	\end{align*}
	where the maximum before the last inequality is bounded by $C(a)$ uniformly in $n\geq 1$ and $|x|\leq a\sqrt{n}$ by the local central limit theorem.
	
	Now, we let $W=X-\tilde{X}$ and we observe that $W$ is a symmetric continuous-time random walk with jump rate $2$. The overlap of $X$ and $\tilde{X}$ then corresponds to the local time of $W$ at $0$. Let us denote the jump times of $W$ by $(\sigma_i)_{i\leq 1}$ and define $T_1=\sigma_1$ and $T_i = \sigma_i - \sigma_{i-1}$ for $i\geq 2$. This way, $(T_i)_{i\geq 1}$ are i.i.d exponential random variables with parameter $2$. Let also $S_i = W_{\sigma_i}$ and observe that $S$ is a simple symmetric discrete-time random walk which is independent of $(T_i)_{i\geq 1}$. If we let $N_t = \max\{i:\, \sigma_i \le t\}$, then
	\begin{eqnarray*}
		L_{s/2}(X,\tilde{X}) = \sum^{N_{s/2}}_{i=1} T_i{\bf 1}_{S_i=0}.
	\end{eqnarray*}
	As on the event $\{X_{s/2},\tilde{X}_{s/2} \leq m\}$ it holds that $N_{s/2} \leq 2m$, it holds that
	\begin{eqnarray*}
		L_{s/2}(X,\tilde{X}) \leq \sum^{2m}_{i=1} T_i{\bf 1}_{S_i=0}.
	\end{eqnarray*}
	and, in particular,
	\begin{eqnarray*}
		E^{\otimes 2}\left[
			e^{2\beta_n^2 L_{s/2}(X,\tilde{X})}
			{\bf 1}_{X_{s/2},\tilde{X}_{s/2} \leq m}
		\right]
		\leq
		E^{\otimes 2}
		\left[
			e^{2 \beta_n^2 \sum^{2m}_{i=1} T_i{\bf 1}_{S_i=0}}
		\right].
	\end{eqnarray*}
	Let $\mathcal{S}=\sigma(S_i:\, i\geq 1)$. Then, using the explicit distribution of the $T_i$'s together with independence,
	\begin{align*}
		E^{\otimes 2}
		\left[
			e^{2 \beta_n^2 \sum^{2m}_{i=1} T_i{\bf 1}_{S_i=0}}
		\right]
		&=
		E^{\otimes 2}
		\left[
		E^{\otimes 2}
		\left[
			e^{2 \beta_n^2 \sum^{2m}_{i=1} T_i{\bf 1}_{S_i=0}}
		\Big{|}
		\mathcal{S}
		\right]
		\right]
		\\
		&=
		E^{\otimes 2}
		\left[
		\prod^{2m}_{i=1}
		E^{\otimes 2}
		\left[
			e^{2 \beta_n^2 T_i{\bf 1}_{S_i=0}}
		\Big{|}
		\mathcal{S}
		\right]
		\right]
		\\
		&=
		E^{\otimes 2}
		\left[
		\prod^{2m}_{i=1}
			\left( 1 - 2 \beta_n^2 {\bf 1}_{S_i=0} \right)^{-1}
		\right]
		\\
		&=
		E^{\otimes 2}
		\left[
			\left( 1 - 2 \beta_n^2 \right)^{-L_{2m}(S)}
		\right]
	\end{align*}
	for $n$ large enough, where $L_{2m}(S)=\sum^{2m}_{i=1}{\bf 1}_{S_i=0}$.
	Summarizing all of the above discussion and using the standard estimate $1-\epsilon\le\epsilon ^\ell$ for all $\ell>0$ and $\epsilon\in\mathbb{R}$, we have 	
\begin{align*}
		E^{\otimes 2}_{t,x,n}\left[
			e^{\beta_n^2 L_{t,x,n}(X,\tilde{X})}
		\right]
		\leq
		C(a)
		E^{\otimes 2}
		\left[
			e^{2\beta_n^2 L_{2m}(S)}
		\right].
	\end{align*}
	The problem is then reduced to a standard pinning estimate. From \cite[Proof of Lemma 4.1]{DGLT}, we can then find two finite constants $C_1=C_1(a,t)$ and $C_2=C_2(a,t)$ such that
	\begin{eqnarray*}
		E^{\otimes 2}_{t,x,n}\left[
			e^{\beta_n^2 L_{t,x,n}(X,\tilde{X})}
		\right]
		\leq
		C_1 e^{C_2 \beta_n^4 m}.
	\end{eqnarray*}
	This proves the first statement.
\end{proof}


\section{A glimpse at the intermediate disorder regime}\label{app:idr}

To provide some heuristics on the intermediate disorder regime aimed at the reader who is not familiar with the topic, we include below two comments.

The first one consists in exhibiting a discrete SPDE for the discrete model. Recall that $Z^{\beta,\theta}(t,n)$ denotes the partition function of the stationary directed polymer model introduced in Section \ref{sec:results-semi-discrete} for $n\geq 1$ and define $Z^{\beta,\theta}(t,0)=e^{\beta B_0(t)+\theta t}$. In this regime, the parameters of the model satisfy the relation $\theta=1 + \tfrac{\beta^2}{2}$.
Let $z(t,j)=e^{-\theta t}Z^{\beta,\theta}(t,j)$. Then, a simple application of It\^o's formula shows that
\begin{eqnarray*}
	dz(t,j)
	=
	\Big(
		z(t,j-1) - z(t,j)
	\Big)\, dt
	+
	\beta z(t,j) \, dB_j(t),
	\quad j\geq 1.
\end{eqnarray*}
This can be seen as a discrete version of the SHE. Indeed, even though the discrete gradient may look odd at first, we recall that the intermediate disorder regime involves the skew scaling $(t,x)\mapsto (nt-x\sqrt{n},nt)$ which mixes the time and space coordinates and leads to a Laplacian in the limit. Recall also that $\beta=\beta_n=n^{-1/4}$ which corresponds to the square root of the space scale and yields the white noise in the SHE as the scaling of the family of Brownian motions.

The second comment appeals to the interpretation of the models as random measures on paths. If the random potential $\wh$ was smooth, the Feynman-Kac formula would allow us to express the solution of the stationary SHE as
\begin{eqnarray*}
	\she^{\text{ST}}(t,x)
	=
	E_{t,x}\left[
		e^{\cB(W_t)}e^{\int^t_0 \wh(t-s,W_{t-s})}ds
	\right],
\end{eqnarray*}
where, under $E_{t,x}$, $W$ is a Brownian motion with $W_0=x$ and we recall that $\cB$ is a two-sided Brownian motion which is independent of $\wh$. The above relation can be formalized as presented, for instance, in \cite{Q-review}. We may think of $W_{t-\cdot}$ as the trajectory of a polymer  based at the space-time point $(t,x)$ and going backwards until it reaches the line $\{(0,x):\, x\in\re\}$ where it collects the boundary condition $e^{\cB}$.

Now, except for a harmless additive constant, the discrete model can be written as
\begin{eqnarray*}
	z(t,n)
	=
	E_{t,n}\left[
		e^{\beta B_0(\sigma_0)}
		e^{\beta H_t(X)}
	\right],
\end{eqnarray*}
where $B_0$ is a two-sided Brownian motion and, under $E_{t,n}$, $X$ is an integer valued totally asymmetric continuous time random walk which jumps from level $j+1$ to level $j$ at rate $1$ until it reaches level $0$ at time $\sigma_0$ and initial position $X_0=n$. The energy of a path $X$ is given by
\begin{eqnarray*}
	H_t(X) = \sum^n_{j=1} B_j(t-\sigma_{j-1},t-\sigma_{j}),
\end{eqnarray*}
where $\sigma_j$ denotes the hitting time of level $j$ for $X$ with the convention $\sigma_{n}=0$, and $\{B_j,\,j\ge 1\}$ are independent two-sided Brownian motions which are independent of $B_0$.  Once again, we can see $X$ as a polymer based at the space-time point $(t,n)$, jumping down until it reaches the line $\{(s,0):\, s\in\re\}$ where it collects the boundary condition $e^{\beta B_0}$.

Now, we may think of the skew scaling as mapping the semi-discrete space-time point $(s,j)$ to the continuum space-time point $(\tfrac{s}{n},\tfrac{j-s}{\sqrt{n}})$. This way, the point $(tn-x\sqrt{n},tn)$ is mapped to $(t-\tfrac{x}{\sqrt{n}},x)$ which becomes $(t,x)$ in the limit. In particular, in the limit, the line $\{(s,0):\, s\in\re\}$ that carries the boundary condition $e^{\beta B_0}$ becomes the line $\{(0,x):\, x\in\re\}$ on which the initial condition for the SHE $e^{\cB}$ is placed.
On the other hand, the walk $X$ becomes a Brownian motion in such a diffusive scaling.

\end{document}